\documentclass[11pt,a4paper]{article}

\usepackage{amsmath}
\usepackage{mathtools}
\usepackage{amsfonts}
\usepackage[mathscr]{eucal}
\usepackage{amsthm}
\usepackage{accents}
\usepackage{amssymb}
\usepackage{enumerate}
\usepackage{cite}
\usepackage{color}
\usepackage[usenames,dvipsnames,svgnames,table]{xcolor}

\DeclareMathOperator{\supp}{supp}

\DeclareMathOperator{\dist}{dist}

\usepackage{pgf,tikz}\usetikzlibrary{matrix, calc, arrows}
\usepackage[bookmarks=true, bookmarksopen=true, bookmarksopenlevel=4]{hyperref}  
\definecolor{myblue}{rgb}{0,0,0.6}     
\hypersetup{pdftitle={Sobolev spaces on subsets of Rn},
            pdfauthor={Chandler-Wilde, Hewett, Moiola},
     colorlinks=true, linkcolor=myblue,  citecolor=myblue, filecolor=myblue,   urlcolor=myblue,  }
\allowdisplaybreaks[4]

\usepackage{fancyhdr}
\begin{document}

\newcommand{\rf}[1]{(\ref{#1})}
\newcommand{\mmbox}[1]{\fbox{\ensuremath{\displaystyle{ #1 }}}}	
\newcommand{\ri}{{\mathrm{i}}}\newcommand{\re}{{\mathrm{e}}}\newcommand{\rd}{\mathrm{d}}
\newcommand{\R}{\mathbb{R}}\newcommand{\Q}{\mathbb{Q}}\newcommand{\N}{\mathbb{N}}
\newcommand{\Z}{\mathbb{Z}}\newcommand{\C}{\mathbb{C}}\newcommand{\K}{{\mathbb{K}}}
\newcommand{\cA}{\mathcal{A}}\newcommand{\cB}{\mathcal{B}}\newcommand{\cC}{\mathcal{C}}
\newcommand{\cS}{\mathcal{S}}\newcommand{\cD}{\mathcal{D}}\newcommand{\cH}{\mathcal{H}}
\newcommand{\cI}{\mathcal{I}}
\newcommand{\cItilde}{\tilde{\mathcal{I}}}
\newcommand{\cIhat}{\hat{\mathcal{I}}}
\newcommand{\cIcheck}{\check{\mathcal{I}}}
\newcommand{\cIstar}{{\mathcal{I}^*}}
\newcommand{\cJ}{\mathcal{J}}\newcommand{\cM}{\mathcal{M}}\newcommand{\cP}{\mathcal{P}}
\newcommand{\cV}{{\mathcal V}}\newcommand{\cW}{{\mathcal W}}
\newcommand{\scrD}{\mathscr{D}}\newcommand{\scrS}{\mathscr{S}}\newcommand{\scrJ}{\mathscr{J}}
\newcommand{\sD}{\mathsf{D}}\newcommand{\sN}{\mathsf{N}}\newcommand{\sS}{\mathsf{S}}
\newcommand{\bs}[1]{\mathbf{#1}}
\newcommand{\bb}{\mathbf{b}}\newcommand{\bd}{\mathbf{d}}\newcommand{\bn}{\mathbf{n}}
\newcommand{\bp}{\mathbf{p}}\newcommand{\bP}{\mathbf{P}}\newcommand{\bv}{\mathbf{v}}
\newcommand{\bx}{\mathbf{x}}\newcommand{\by}{\mathbf{y}}\newcommand{\bz}{{\mathbf{z}}}
\newcommand{\bxi}{\boldsymbol{\xi}}\newcommand{\boldeta}{\boldsymbol{\eta}}
\newcommand{\ts}{\tilde{s}}\newcommand{\tGamma}{{\tilde{\Gamma}}}
\newcommand{\done}[2]{\dfrac{d {#1}}{d {#2}}}
\newcommand{\donet}[2]{\frac{d {#1}}{d {#2}}}
\newcommand{\pdone}[2]{\dfrac{\partial {#1}}{\partial {#2}}}
\newcommand{\pdonet}[2]{\frac{\partial {#1}}{\partial {#2}}}
\newcommand{\pdonetext}[2]{\partial {#1}/\partial {#2}}
\newcommand{\pdtwo}[2]{\dfrac{\partial^2 {#1}}{\partial {#2}^2}}
\newcommand{\pdtwot}[2]{\frac{\partial^2 {#1}}{\partial {#2}^2}}
\newcommand{\pdtwomix}[3]{\dfrac{\partial^2 {#1}}{\partial {#2}\partial {#3}}}
\newcommand{\pdtwomixt}[3]{\frac{\partial^2 {#1}}{\partial {#2}\partial {#3}}}
\newcommand{\bnabla}{\boldsymbol{\nabla}}
\newcommand{\eps}{\varepsilon}
\newcommand{\im}[1]{{\rm Im}\left[#1\right]}
\newcommand{\ol}[1]{\overline{#1}}
\newcommand{\ord}[1]{\mathcal{O}\left(#1\right)}
\newcommand{\oord}[1]{o\left(#1\right)}
\newcommand{\Ord}[1]{\Theta\left(#1\right)}
\newcommand{\norm}[2]{\left\|#1\right\|_{#2}}
\newtheorem{thm}{Theorem}[section]
\newtheorem{lem}[thm]{Lemma}
\newtheorem{defn}[thm]{Definition}
\newtheorem{prop}[thm]{Proposition}
\newtheorem{cor}[thm]{Corollary}
\newtheorem{rem}[thm]{Remark}
\newtheorem{conj}[thm]{Conjecture}
\newtheorem{ass}[thm]{Assumption}
\newtheorem{example}[thm]{Example} 
\newcommand{\tH}{\widetilde{H}}
\newcommand{\Hze}{H_{\rm ze}} 	
\newcommand{\uze}{u_{\rm ze}}		
\newcommand{\dimH}{{\rm dim_H}}
\newcommand{\dimB}{{\rm dim_B}}
\newcommand{\IntClosOm}{\mathrm{int}(\overline{\Omega})}
\newcommand{\IntClosOmOne}{\mathrm{int}(\overline{\Omega_1})}
\newcommand{\IntClosOmTwo}{\mathrm{int}(\overline{\Omega_2})}
\newcommand{\Ccomp}{C^{\rm comp}}
\newcommand{\tCcomp}{\tilde{C}^{\rm comp}}
\newcommand{\uC}{\underline{C}}
\newcommand{\utC}{\underline{\tilde{C}}}
\newcommand{\oC}{\overline{C}}
\newcommand{\otC}{\overline{\tilde{C}}}
\newcommand{\capcomp}{{\rm cap}^{\rm comp}}
\newcommand{\Capcomp}{{\rm Cap}^{\rm comp}}
\newcommand{\tcapcomp}{\widetilde{{\rm cap}}^{\rm comp}}
\newcommand{\tCapcomp}{\widetilde{{\rm Cap}}^{\rm comp}}
\newcommand{\hcapcomp}{\widehat{{\rm cap}}^{\rm comp}}
\newcommand{\hCapcomp}{\widehat{{\rm Cap}}^{\rm comp}}
\newcommand{\tcap}{\widetilde{{\rm cap}}}
\newcommand{\tCap}{\widetilde{{\rm Cap}}}
\newcommand{\ccap}{{\rm cap}}
\newcommand{\ucap}{\underline{\rm cap}}
\newcommand{\uCap}{\underline{\rm Cap}}
\newcommand{\cCap}{{\rm Cap}}
\newcommand{\ocap}{\overline{\rm cap}}
\newcommand{\oCap}{\overline{\rm Cap}}
\DeclareRobustCommand
{\mathringbig}[1]{\accentset{\smash{\raisebox{-0.1ex}{$\scriptstyle\circ$}}}{#1}\rule{0pt}{2.3ex}}
\newcommand{\cirH}{\mathringbig{H}}
\newcommand{\cirHs}{\mathringbig{H}{}^s}
\newcommand{\cirHt}{\mathringbig{H}{}^t}
\newcommand{\cirHm}{\mathringbig{H}{}^m}
\newcommand{\cirHzero}{\mathringbig{H}{}^0}
\newcommand{\deO}{{\partial\Omega}}
\newcommand{\OO}{{(\Omega)}}
\newcommand{\Rn}{{(\R^n)}}
\newcommand{\Id}{{\mathrm{Id}}}
\newcommand{\gap}{\mathrm{Gap}}
\newcommand{\ggap}{\mathrm{gap}}
\newcommand{\isom}{{\xrightarrow{\sim}}}
\newcommand{\half}{{1/2}}
\newcommand{\mhalf}{{-1/2}}
\newcommand{\baro}{{\overline{\Omega}}} 
\newcommand{\inter}{{\mathrm{int}}}

\newcommand{\Hsp}{H^{s,p}}
\newcommand{\Htq}{H^{t,q}}
\newcommand{\tHsp}{{{\widetilde H}^{s,p}}}
\newcommand{\SP}{\ensuremath{(s,p)}}
\newcommand{\Xsp}{X^{s,p}}
\newcommand{\dd}{{d}}
\newcommand{\pp}{{p_*}}
\newcommand{\Rnn}{\R^{n_1+n_2}}
\newcommand{\Tr}{{\mathrm{Tr}}}

\title{On the maximal Sobolev regularity\\ of distributions supported by subsets of Euclidean space}
\author{D.\ P.\ Hewett\footnotemark[1], A.\ Moiola\footnotemark[2]}

\renewcommand{\thefootnote}{\fnsymbol{footnote}}
\footnotetext[1]{Mathematical Institute, University of Oxford, Radcliffe Observatory Quarter, Woodstock Road, Oxford, OX2 6GG, UK. 
Current address: Department of Mathematics,
University College London,
Gower Street, London, WC1E 6BT, UK.
E-mail: \texttt{d.hewett@ucl.ac.uk}
}
\footnotetext[2]{Department of Mathematics and Statistics, University of Reading, Whiteknights PO Box 220, Reading RG6 6AX, UK. Email: 
\texttt{a.moiola@reading.ac.uk}}
\maketitle
\renewcommand{\thefootnote}{\arabic{footnote}}

\begin{abstract} 
This paper concerns the following question: given a subset $E$ of $\R^n$ with empty interior and an integrability parameter $1<p<\infty$, what is the maximal regularity $s\in\R$ for which there exists a non-zero distribution in the Bessel potential Sobolev space $H^{s,p}(\R^n)$ that is supported in $E$?
For sets of zero Lebesgue measure we apply well-known results on set capacities from potential theory to characterise the maximal regularity in terms of the Hausdorff dimension of $E$, sharpening previous results. 
Furthermore, we provide a full classification of all possible maximal regularities, as functions of $p$, together with the sets of values of $p$ for which the maximal regularity is attained, and construct concrete examples for each case.
Regarding sets with positive measure, for which the maximal regularity is non-negative, we present new lower bounds on the maximal Sobolev regularity supported by certain fat Cantor sets, which we obtain both by capacity-theoretic arguments, and by direct estimation of the Sobolev norms of characteristic functions.
We collect several results characterising the regularity that can be achieved on certain special classes of sets, such as $d$-sets, boundaries of open sets, and Cartesian products, of relevance for applications in differential and integral equations.
\\[3mm]
\\[3mm]
\end{abstract}

\section{Introduction}\label{sec:intro}
This paper concerns the following question \textbf{Q}:\\ 
\emph{Given a subset $E$ of $\R^n$ with empty interior,  an integrability parameter $1<p<\infty$, and a regularity parameter $s\in\R$, does there exist a non-zero distribution in the Bessel potential Sobolev space $H^{s,p}(\R^n)$ 
which is supported in $E$?}

This question has arisen repeatedly in the course of the first author's recent investigations \cite{CoercScreen,Ch:13,CoercScreen2,ScreenPaper} into the analysis of acoustic scattering by planar screens with rough (e.g.\ fractal) boundaries. Indeed, for such scattering problems one of the factors determining the unique solvability of the Helmholtz equation boundary value problems (BVPs) with Dirichlet or Neumann boundary conditions, at least as they are classically posed, and the associated boundary integral equation (BIE) formulations, is whether the boundary of the screen (the screen being viewed as a relatively open subset of the plane) can support non-zero elements of $H^{\pm 1/2,2}(\R^2)$ \cite{CoercScreen,ScreenPaper}.

More generally, the question \textbf{Q} pertains to a number of other fundamental questions about function spaces on subsets of $\R^n$ defined in terms of the spaces $H^{s,p}(\R^n)$. 
We give a simple illustration of this in Proposition \ref{thm:Hs_equality_closed} below, where we show how \textbf{Q} is related to the question of whether $H^{s,p}_{F_1}=H^{s,p}_{F_2}$ for closed sets $F_1\neq F_2\subset \R^n$. 
In \cite{CoercScreen,ChaHewMoi:13,Hs0paper}, where our focus is on the case $p=2$, we demonstrate the relevance of \textbf{Q} for understanding when $H^{s,2}_0(\Omega)=H^{s,2}(\Omega)$ and when $\tH^{s,2}(\Omega)=H^{s,2}_{\ol\Omega}$, for a given open set $\Omega\subset \R^n$, and also for understanding when $\tH^{s,2}(\Omega_1)=\tH^{s,2}(\Omega_2)$ for open sets $\Omega_1\neq\Omega_2\subset \R^n$.
(Here, for closed $F\subset\R^n$, $H^{s,p}_F:=\{u\in H^{s,p}(\R^n):\supp u\subset F\}$, and for open $\Omega\subset \R^n$, $H^{s,p}(\Omega):=\{u|_\Omega:u\in H^{s,p}(\R^n)\}$, $H^{s,p}_0(\Omega) = \overline{C^\infty_0(\Omega)}^{H^{s,p}(\Omega)}$ and $\tH^{s,p}(\Omega) = \overline{C^\infty_0(\Omega)}^{H^{s,p}(\R^n)}$.)

Upon consulting the classical function space literature we found a number of disparate partial results relevant to the question \textbf{Q} (in particular we note \cite{Ca:67,Li:67a,Li:67b,Po:72,Po:72a,Ad:74,AdHe,Triebel97FracSpec,Tri:08,Maz'ya}), but no single convenient and up-to-date reference in which these results are collected in a form easily accessible to applied and numerical analysts. The aim of this paper is to provide such a reference, which we hope will be of use to those interested in problems involving PDEs and integral equations on rough (i.e.,\ non-Lipschitz) domains. But this is not simply a review paper. 
We also present a number of apparently new results, along with a range of concrete examples and counterexamples illustrating them. The key new results we contribute include:
\begin{itemize}
\item a sharpening of the relationship between maximal Sobolev regularity and fractal dimension (cf.\ Theorem \ref{thm:NullityHausdorff} and Remark \ref{rem:TriebelHausdorff});
\item a complete characterisation of all possible maximal regularity behaviours for sets with zero Lebesgue measure (cf.\ Corollary \ref{cor:NullitySets} and the concrete examples in Theorem \ref{thm:CantorZoo});
\item new results on the Sobolev regularity of the characteristic functions of certain fat Cantor sets with positive Lebesgue measure (Propositions \ref{prop:FatCantorCapacity}--\ref{prop:FatCantor}). 
\end{itemize}
While the paper does not provide a definitive answer to \textbf{Q} in its full generality, we hope that the results we provide, along with the open questions that we pose, will stimulate further research.

Function space experts might correctly observe that the question \textbf{Q} could be posed in a much more general setting, for instance in the context of the Besov and Triebel--Lizorkin spaces $B^s_{pq}(\R^n)$ and $F^s_{pq}(\R^n)$ \cite{Triebel83ThFS,Maz'ya,AdHe,Triebel97FracSpec}, of which $H^{s,p}(\R^n)=F^s_{p2}(\R^n)$ is a special case. 
Our decision to restrict attention to the classical Bessel potential Sobolev spaces $H^{s,p}(\R^n)$ 
(sometimes referred to as ``fractional Sobolev spaces'', ``Liouville spaces'' or ``Lebesgue spaces'') 
is made for two reasons. First, it allows a relatively simple and accessible presentation: the proofs of many of our results make use of classical nonlinear potential theoretic results on set capacities and Bessel potentials already available e.g.\ in \cite{AdHe,Maz'ya}, allowing us to avoid any discussion of more intricate theories such as atomic and quarkonial decompositions which are typically employed in the modern function space literature to analyse the spaces $B^s_{pq}(\R^n)$ and $F^s_{pq}(\R^n)$ \cite{Triebel83ThFS,Maz'ya,AdHe,Triebel97FracSpec}. 
Second, the spaces $H^{s,p}(\R^n)$ are sufficient for a very large part of the study of linear elliptic BVPs and BIEs, which are the focus of attention for example in the classic monographs \cite{LiMaI} and \cite{ChPi} and in the much more recent book by McLean \cite{McLean} that has become the standard reference for the theory of BIE formulations of BVPs for strongly elliptic systems. In such applications the focus is usually on the case $p=2$, but since the potential theoretic results we cite from \cite{AdHe,Maz'ya} are valid for any $1<p<\infty$, it seems natural to present results for this general case wherever possible.

The structure of the paper is as follows. In \S\ref{sec:MainResults} we review some basic definitions, introduce the concepts of ``$(s,p)$-nullity'' and the ``nullity threshold'' of a set $E\subset\R^n$ (which will provide a framework within which to study question \textbf{Q}), and state our main results. Sets with zero and positive Lebesgue measure require different analyses, we study them in \S\ref{subsec:ResultsSL0} and \S\ref{subsec:ResultsSG0} respectively. In \S\ref{sec:Capacity} we collect a number of results relating to certain set capacities from nonlinear potential theory, which we use to prove the results of \S\ref{sec:MainResults}. In \S\ref{sec:Domains} we provide concrete examples and counterexamples to illustrate our general results. In \S\ref{sec:conclusion} we offer some conclusions and highlight the key open problems arising from our investigations.

\tableofcontents
\addcontentsline{toc}{subsection}{Contents}

\section{Main results}
\label{sec:MainResults}

\subsection{Preliminaries}
Before stating our main results we fix our notational conventions. Given $n\in \N$, let $\scrD=\scrD(\R^n)$ denote the space of compactly supported (real- or complex-valued) smooth test functions on~$\R^n$. 
For any open set $\Omega\subset \R^n$ let
$\scrD(\Omega):=\{u\in\scrD:\supp{u}\subset\Omega\}$, let $\scrD^*(\Omega)$ denote the associated space of distributions (anti-linear continuous functionals on $\mathscr{D}(\Omega)$), and let $L^1_{\rm loc}(\Omega)\subset \scrD^*(\Omega)$ denote the space of locally integrable functions on $\Omega$; 
for brevity we write $\scrD^*=\scrD^*(\R^n)$ and $L^1_{\rm loc} = L^1_{\rm loc}(\R^n)$. Similarly for $1<p<\infty$ we write $L^p=L^p\Rn$ and $L^p_{\rm loc}=L^p_{\rm loc}\Rn$, and denote by $p'$ the H\"older conjugate of $p$, i.e.\ the number $1<p'<\infty$ such that $1/p + 1/p'=1$. For any set $E\subset\R^n$ we denote the complement of $E$ by $E^c:=\R^n\setminus E$, and the closure of $E$ by $\overline{E}$. Let $\emptyset$ denote the empty set. Given $\bx\in\R^n$ and $\eps>0$ let $B_\eps(\bx)$ denote the open ball of radius $\eps$ centred at $\bx$. 
Let $\mathscr{S}$ denote the Schwartz space of rapidly decaying smooth test functions on $\R^n$, and $\mathscr{S}^*$ the dual space of tempered distributions (anti-linear continuous functionals on $\mathscr{S}$).
For $u\in \mathscr{S}$ we define the Fourier transform $\hat{u}={\cal F} u\in \mathscr{S}$ and its inverse $\check{u}={\cal F}^{-1} u\in \mathscr{S}$ by 
\begin{align*}
\hat{u}(\bxi):= \frac{1}{(2\pi)^{n/2}}\int_{\R^n}\re^{-\ri \bxi\cdot \bx}u(\bx)\,\rd \bx , \;\; \bxi\in\R^n, \quad
\check{u}(\bx) := \frac{1}{(2\pi)^{n/2}}\int_{\R^n}\re^{\ri \bxi\cdot \bx}u(\bxi)\,\rd \bxi , \;\;\bx\in\R^n.
\end{align*}
We define the Bessel potential operator $\cJ_s$ on $\mathscr{S}$, for $s\in\R$, by $\cJ_s := {\cal F}^{-1}\cM_s{\cal F}$, where $\cM_s$ represents multiplication by $(1+|\bxi|^2)^{s/2}$.
We extend these definitions to $\mathscr{S}^*$ in the usual way: 
\begin{align*}
\hat{u}(v) := u(\check{v}),\quad
\check{u}(v) := u(\hat{v}),\quad \cM_su(v) := u(\cM_s v), \quad
(\cJ_s u)(v) := u(\cJ_s v),
\qquad u\in \mathscr{S}^*, \, v\in \mathscr{S},
\end{align*}
and note that for $u\in \mathscr{S}^*$ it holds that $\widehat{\cJ_s u} = \cM_s\hat{u}$. 

For $s\in\R$ and $1<p<\infty$ the Bessel potential Sobolev space $H^{s,p}(\R^n)$ (abbreviated throughout to $H^{s,p}$, except in Appendix \ref{s:TensorProduct} where different dimensions $n$ are considered) is defined by
\begin{align*}
H^{s,p}:=\left\{u\in \mathscr{S}^* \,:\, \cJ_s u \in L^p\right\},
\quad \textrm{with }
\|u\|_{H^{s,p}} := \|\cJ_s u\|_{L^p}.
\end{align*}
Note that in the special case $p=2$, the norm $\|u\|_{H^{s,2}}$ can be realised using Plancherel's theorem as
\begin{align}
\label{eqn:Hs2NormFourier}
\|u\|_{H^{s,2}} = \left(\int_{\R^n}(1+|\bxi|^2)^{s}|\hat{u}(\bxi)|^2\,\rd \bxi\right)^{1/2}.
\end{align}

Other commonly used notation for $H^{s,p}$ includes $H^s_p$ (cf.\ \cite{Maz'ya,Triebel83ThFS}) and $L^{s,p}$ (cf.\ \cite{AdHe}). In relation to the wider function space literature we recall that (cf.\ e.g.\ \cite[\S2.5.6]{Triebel83ThFS}) $H^{s,p}=F^s_{p2}$ with equivalent norms, where $F^s_{pq}$ are the Triebel--Lizorkin spaces. For $s\geq0$, let $W^{s,p}\subset L^p$ be the classical Sobolev--Slobodeckij--Gagliardo space defined in terms of weak derivatives (cf.\ e.g.\ \cite[pp.~73--74]{McLean}). Then for $s\in \N_0$ it holds that $H^{s,p}=W^{s,p}$ with equivalent norms \cite[\S2.3.5]{Triebel83ThFS} (in particular, $H^{0,p}=L^p$ with equal norms). For $p=2$ this result extends to all $s\geq 0$ \cite[Theorem 3.16]{McLean}. For $p\neq 2$ and $0<s\notin\N$ it holds that $W^{s,p}=B^s_{pp}$ with equivalent norms \cite[\S2.2.2]{Triebel83ThFS} (here $B^s_{pq}$ are the Besov spaces), so that (by \cite[\S2.3.2]{Triebel83ThFS} and \cite[Theorem 2.12(c)]{Triebel78ITFSDO}) $W^{s,p}\subsetneqq H^{s,p}$ for $1<p<2$ and 
$H^{s,p}\subsetneqq W^{s,p}$ for $2<p<\infty$.

We recall some basic properties of $H^{s,p}$ that will be useful later. 
It is well known that $\scrD$ is dense in $H^{s,p}$, 
and that the following embeddings are continuous with dense image: \cite[\S2.7.1]{Triebel83ThFS} 
\begin{align}
\label{eq:Embedding}
H^{t,q} \subset H^{s,p}, \qquad 1<q\le p<\infty, \quad t-s\geq n\left(\frac{1}{q}-\frac{1}{p}\right)\ge0.
\end{align}
For distributions with compact support a more general embedding result holds. 
Given a closed set $F\subset \R^n$, define the closed subspace $H^{s,p}_F\subset H^{s,p}$ by
\begin{equation*} 
H_F^{s,p} :=\big\{u\in H^{s,p}: \supp(u) \subset F\big\},
\end{equation*}
where the support of a distribution $u\in\scrD$ is defined in the usual way, namely as the largest closed subset $\Lambda\subset\R^n$ for which $u(\phi)=0$ for every $\phi\in\scrD(\Lambda^c)$ (see e.g.\ \cite[p.\ 66]{McLean}). 
Then, since pointwise multiplication by a fixed element of $\scrD$ defines a bounded linear operator from $H^{s,q}$ to $H^{s,p}$ for any $1<p\leq q<\infty$ (see e.g.\ \cite[Lemma 4.6.2]{Triebel78ITFSDO}), for any compact $K\subset \R^n$ the following embedding is continuous (in particular this holds for $s=t$ and $1<p\le q<\infty$):
\begin{align}
\label{eq:EmbeddingCompact}
H^{t,q}_K \subset H^{s,p}_K, \quad t-s\geq \max\left\lbrace n\left(\frac{1}{q}-\frac{1}{p}\right),0 \right\rbrace.
\end{align}
The dual space of $H^{s,p}$ can be isometrically realised as the space $H^{-s,p'}$  
by the duality pairing
\begin{align*}
\langle u,v \rangle_{H^{-s,p'}\times H^{s,p}} = \langle\cJ_{-s}u,\cJ_s v \rangle_{L^{p'}\times L^p},
\end{align*}
which in the special case $p=2$ can be realised using Plancherel's theorem as
\begin{align*}
\langle u,v \rangle_{H^{-s,2}\times H^{s,2}} = \int_{\R^n}\hat{u}(\bxi)\overline{\hat{v}(\bxi)}\,\rd \bxi.
\end{align*}

When $s>n/p$, elements of $H^{s,p}$ are continuous functions by the Sobolev embedding theorem \cite[Theorem 1.2.4]{AdHe}. At the other extreme, for any $\bx_0\in\R^n$, the Dirac delta function, defined as $\delta_{\bx_0}(\phi)=\overline{\phi(\bx_0)}$ for $\phi\in\scrD\Rn$ to fit our convention of using anti-linear functionals, satisfies
\begin{equation}\label{eq:delta}
\delta_{\bx_0}\in H^{s,p}\qquad \text{if and only if}\qquad s<-n/p'.
\end{equation}
Finally, we note that part (d) of Theorem 1 in \cite[\S2.4.2]{Triebel78ITFSDO} allows the spaces $\Hsp$ to be arranged in interpolation scales.
For $s_0,s_1\in \R$, $1<p_0,p_1<\infty$ and $0<\theta<1$,  
\begin{equation}\label{eq:Interpolation}
\textrm{if }\quad  s=(1-\theta)s_0 + \theta s_1 \quad \textrm{and } \quad  \frac1p=\frac{1-\theta}{p_0}+\frac\theta{p_1},
\quad \text{then} \quad  \Hsp=[H^{s_0,p_0},H^{s_1,p_1}]_\theta,
\end{equation}
where $[\cdot,\cdot]_\theta$ denotes the space of exponent $\theta$ obtained with the complex interpolation method (see \cite[\S1.9]{Triebel78ITFSDO}), and equality of spaces holds with equivalent norms.
Thus, if the spaces $\Hsp$ are represented by points in the $(1/p,s)$-plane, then straight segments constitute interpolation scales. 

\subsection{\texorpdfstring{$(s,p)$}{(s,p)}-Nullity}
\label{subsec:nullity}
We now introduce the concept of $(s,p)$-nullity, which will be the focus of our studies.
\begin{defn}
\label{def:nullity}
Given $1<p<\infty$ and $s\in\R$ we say that a set $E\subset\R^n$ is \emph{$(s,p)$-null} if $H^{s,p}_{F}=\{0\}$ for every closed set $F\subset E$.
\end{defn}
In other words, a set $E\subset\R^n$ is $(s,p)$-null if and only if there are no non-zero elements of $H^{s,p}$ supported in $E$. 

\begin{rem}
\label{rem:Compact}
Clearly, if $F$ is closed then $F$ is $(s,p)$-null if and only if $H^{s,p}_F=\{0\}$. 
Note also that the Definition \ref{def:nullity} can be equivalently stated with ``closed'' replaced by ``compact''. Indeed, if $0\neq u\in H^{s,p}$ with $\supp u\subset E$ then $0\neq \phi u\in H^{s,p}$ is compactly supported in $E$ for any $\phi\in\scrD$ such that $\phi(\bx)\neq 0$ for some $\bx\in\supp u$ (cf.\ the proof of Proposition \ref{thm:NullityOfUnions}\rf{ll}).
\end{rem}

While our terminology ``$(s,p)$-null'' appears to be new, the concept it describes has been studied previously, apparently first by H\"{o}rmander and Lions in relation to properties of Sobolev spaces normed by Dirichlet integrals \cite{HoLi:56}, and then subsequently by a number of other authors in relation to the removability of singularities for elliptic partial differential operators \cite{Li:67a,Maz'ya}, and to the approximation of functions by solutions of the associated elliptic PDEs \cite{Po:72}.
For integer $s<0$ the concept of $(s,p)$-nullity is referred to (in the special case $p=2$) as $(-s)$-polarity in \cite[D\'efinition~2]{HoLi:56}, ``$p'$-$(-s)$ polarity'' in \cite{Li:67a} and ``$(p',-s)$-polarity'' in \cite[\S 13.2]{Maz'ya}. 
A related notion is discussed in the more general context of the spaces $B^s_{pq}$ in \cite[\S17]{Triebel97FracSpec} (see Remark \ref{rem:TriebelHausdorff}). 
For $s>0$, our notion of $(s,p)$-nullity is closely related to the concept of ``sets of uniqueness'' considered in \cite[\S11.3]{AdHe} and \cite[p.~692]{Maz'ya} (for integer $s$); this relationship is discussed in \S\ref{subsec:SOU}. 
For $s>0$ and $E$ with empty interior, the concept of nullity coincides with the concept of $(s,p)$-stability, discussed in \cite[\S11.5]{AdHe}. 

The reason why Maz'ya \cite{Maz'ya} uses two different terminologies (polarity and set of uniqueness) for the positive and negative order spaces is not made clear in \cite{Maz'ya}, but this may be due to the fact that Maz'ya works primarily with the spaces $W^{s,p}$, where the positive order spaces are defined using weak derivatives, and the negative order spaces are defined by duality. By contrast, in the Bessel potential framework of the current paper, the spaces $H^{s,p}$ are defined in the same way for all $s\in\R$, so that it seems natural to define the notion of ``negligibility'' in the same way for all $s\in\R$. Our decision to introduce the terminology ``$(s,p)$-nullity'' instead of using ``$(p',-s)$-polarity'' was made simply for clarity (personally we find it more natural to say that a set which does not support an $H^{s,p}$ distribution is ``$(s,p)$-null'' rather than ``$(p',-s)$-polar''). But the difference is purely semantic, so readers familiar with the concept of polarity may read ``$(p',-s)$-
polar'' for ``$(s,p)$-null'' throughout.

The following lemma collects some elementary facts about $(s,p)$-nullity.
\begin{lem}
\label{lem:nullity1}
Let $1<p,q<\infty$, $s,t\in\R$ and $E\subset\R^n$.
 \begin{enumerate}[(i)]
\item \label{aa}If $E$ is $(s,p)$-null and $E'\subset E$ then $E'$ is $(s,p)$-null.
\item \label{bb}If $E$ is $(s,p)$-null and $t\geq s + \max\{n(1/q-1/p),0\}$ then $E$ is $(t,q)$-null.
\item \label{cc}If $E$ is $(s,p)$-null then $E$ has empty interior.
\item \label{dd}If $s>n/p$ then $E$ is $(s,p)$-null if and only if $E$ has empty interior.
\item \label{qq} $E$ is $(0,p)$-null if and only if $\underline{m}(E)=0$, where $\underline{m}$ denotes inner Lebesgue measure (cf.\ Remark~\ref{rem:CapMeasure}).
\item \label{jj}For $s<-n/p'$ there are no non-empty $(s,p)$-null sets.
\item \label{nn} Let $1<p_0,p_1<\infty$ and $s_0,s_1\in\R$. If there exists $0\neq u\in H^{s_0,p_0}\cap H^{s_1,p_1}$ with $\supp u\subset E$, then $E$ is not $(s,p)$-null for $(s,p)$ defined as in \rf{eq:Interpolation}, for every $0<\theta<1$. 
\end{enumerate}
\end{lem}
\begin{proof}
\rf{aa} and \rf{bb} follow straight from the definition of $(s,p)$-nullity, the standard embeddings \rf{eq:Embedding} and \rf{eq:EmbeddingCompact}, and the boundedness on $H^{s,p}$ of pointwise multiplication by elements of $\scrD$.
\rf{cc} If $E$ has non-empty interior then one can trivially construct a non-zero element of $\scrD\subset H^{s,p}$ supported inside $E$. 
\rf{dd} follows from \rf{cc} and the Sobolev embedding theorem. 
\rf{qq} follows from the fact that a closed set supports a non-zero $L^p$ function if and only if it has non-zero measure.
\rf{jj} follows from \eqref{eq:delta}, and \rf{nn} follows from \rf{eq:Interpolation}.
\end{proof}

Lemma \ref{lem:nullity1} immediately implies the following proposition.
\begin{prop}
\label{prop:NullityThreshold}
Fix $1<p<\infty$. For every $E\subset\R^n$ with empty interior there exists 
$$s_E(p)\in[-n/p',n/p]$$
such that $E$ is $(s,p)$-null for $s>s_E(p)$ and not $(s,p)$-null for $s<s_E(p)$. 
We call $s_E(p)$ the \emph{nullity threshold} of $E$ for the integrability parameter $p$. 
\end{prop}

Our aim in this paper is to investigate the following three questions:
\begin{itemize}
\item[\textbf{Q1:}] Given $1<p<\infty$ and $E\subset \R^n$ with empty interior, can we determine $s_E(p)$?
\item[\textbf{Q2:}] For which functions $f:(1,\infty)\to[-n,n]$ does there exist $E\subset\R^n$ such that $f(p)=s_E(p)$ for all $p\in(1,\infty)$?
\item[\textbf{Q3:}]
Under what conditions on $E$ and $p$ is $E$ ``threshold null'' (i.e.\ $(s_E(p),p)$-null)? 
\end{itemize}
Our (partial) answers to these questions are summarised in \S\ref{sec:conclusion}. To state some of our results it will be useful to introduce the ``nullity set'' and ``threshold nullity set'' of a set $E\subset \R^n$, defined by
\begin{align}
\label{eqn:NullitySet}
\mathcal N_E&:=\big\{(s,p)\in\R\times(1,\infty)\text{ s.t.\ $E$ is $(s,p)$-null}\big\},\\
\label{eqn:ThresholdNullitySet}
\mathcal T_E&:=\big\{p\in (1,\infty) \text{ s.t.\ $E$ is $(s_E(p),p)$-null}\big\}.
\end{align}

Our attempts to answer questions \textbf{Q1}--\textbf{Q3} will make extensive use of the relationship between $(s,p)$-nullity and certain set capacities from classical potential theory. The following key theorem is stated in \cite[Theorem 13.2.2]{Maz'ya} for the case where $s$ is a negative integer, but Maz'ya's proof in fact works for all $s\in\R$. 
We note that this result is actually a special case of a more general result proved in \cite[Lemma 1]{Li:67a} (where the result is attributed to Grusin \cite{Gr:62}). 
The inner capacity $\underline{\rm Cap}$ appearing in the theorem is defined in \S\ref{sec:Capacity} below.
\begin{thm}[{\cite[Theorem 13.2.2]{Maz'ya}, \cite[Lemma 1]{Li:67a}}]
\label{thm:NullityCapEquiv}
Let $1<p<\infty$ and $s\in\R$. Then $E\subset\R^n$ is $(s,p)$-null if and only if $\underline{\rm Cap}{}_{-s,p'}(E)=0$.
\end{thm}

Maz'ya's proof goes via the following intermediate result, which we state for future reference, since it provides another useful characterisation of $(s,p)$-nullity for closed sets.
\begin{thm}[{\cite[Theorem 13.2.1]{Maz'ya}}]
\label{thm:NullityDensityEquiv}
Let $1<p<\infty$ and $s\in\R$. Then a closed set $F\subset\R^n$ is $(s,p)$-null if and only if $\scrD(F^c)$ is dense in $H^{-s,p'}$.
\end{thm}

Theorem \ref{thm:NullityCapEquiv}, combined with the classical potential theoretic results developed e.g.\ in \cite{AdHe,Maz'ya}, and summarised in \S\ref{sec:Capacity} below, will underpin the proofs of many of our results about $(s,p)$-nullity, including part \rf{mm} of the following proposition, the proof of which is given in \S\ref{sec:Capacity}.
\begin{prop}
\label{thm:NullityOfUnions}
Let $1<p<\infty$ and $s\in\R$.
\begin{enumerate}[(i)]
\item \label{ll} 
If $E,F\subset \R^n$ are both $(s,p)$-null and $F$ has no limit points in $E\setminus F$ (which holds, for example, if $F$ is closed),
then $E\cup F$ is $(s,p)$-null. 
In particular, a finite union of $(s,p)$-null closed sets is $(s,p)$-null.
\item \label{mm} 
For $s\leq 0$, a countable union of $(s,p)$-null Borel sets is $(s,p)$-null.
\end{enumerate}
\end{prop}
\begin{proof}
\rf{ll} We argue by contrapositive. 
Suppose that $E\cup F$ is not $(s,p)$-null, i.e.\ there exists a non-zero $u\in H^{s,p}$ with $\supp{u}\subset E\cup F$. 
Then if $\supp{u}\subset F$, $F$ is not $(s,p)$-null and we are done. If not, there exists $\bx\in\supp{u}\cap(E\setminus F)$, and, 
since $F$ has no limit points in $E\setminus F$, $\eps:=\dist(\bx,F)>0$. 
Let $\phi\in\scrD(B_\eps(\bx))$ with $\phi(\bx)\neq 0$.
Then $0\neq\phi u\in H^{s,p}$ with $\supp{\phi u}\subset E$, so $E$ is not $(s,p)$-null.

That $\phi u \neq 0$ follows from the fact that if $u\in \scrD^*$ and $\phi\in \scrD$,
and if there exists $\bx\in\supp{u}$ such that $\phi(\bx)\neq 0$, then $\phi u\neq 0$ as a distribution on $\R^n$.
To see this, 
let $\eps>0$ be such that $\phi$ is non-zero in $B_\eps(\bx)$.
Then, since $\bx\in\supp{u}$, $u|_{B_\eps(\bx)}\neq 0$ and so $u(\psi)\neq0$ for some $\psi\in \scrD(B_\eps(\bx))$.
But then, defining $\varphi\in \scrD$ by $\varphi(\bx):=\psi/\phi$, for $\bx\in B_\eps(\bx)$, and $\varphi(\bx):=0$ otherwise, we have $(\phi u)(\varphi) = u(\psi)\neq 0$, so $\phi u$ is non-zero as claimed.
\end{proof}

\begin{rem}\label{rem:ExampleRQ}
Regarding part \rf{ll} of Proposition \ref{thm:NullityOfUnions}, it is natural to ask to what extent the assumption on $F$ can be weakened.
Certainly the result does not extend to general Borel $F$ when $s>n/p$. For a simple counterexample, let $E_1$ denote the elements of the open unit ball $B=B_1(\mathbf{0})$ which have at least one rational coordinate, and let $E_2=B\setminus E_1$. Then for $s>n/p$ both $E_1$ and $E_2$ are $(s,p)$-null, since they both have empty interior. But $E_1\cup E_2=B$, which is not $(s,p)$-null for any $s\in\R$, since it has non-empty interior.

This example also shows that part \rf{mm} of Proposition \ref{thm:NullityOfUnions} does not hold for all $s\in\R$.
Determining the maximal $s\in[0,n/p]$ such that Proposition \ref{thm:NullityOfUnions}\rf{mm} holds appears to be an \textbf{open problem}.
\end{rem}

The following proposition gives bounds on the nullity threshold of Cartesian products, derived from Propositions~\ref{prop:TensorPositive} (the lower bound) and \ref{prop:Trace} (the upper bound) in the Appendix. More general results can be derived for Cartesian products of more than two sets, but we do not present them here. 
The assumption that $E_1,E_2$ are Borel is needed only for the upper bound in the case $m(E_1\times E_2)=0$. 
\begin{prop}\label{cor:TensorProduct}
Let $n_1,n_2\in\N$, $1<p<\infty$, and let $E_1\subset \R^{n_1}$ and $E_2\subset \R^{n_2}$ be Borel.
Then the nullity threshold of the Cartesian product $E_1\times E_2\subset\Rnn$ satisfies: 
\begin{align}
\label{eq:ProductBounds}
s_-(p) \le s_{E_1\times E_2}(p)\le s_+(p),
\end{align}
where
\begin{align*}
\label{}
s_-(p)&:=\min\big\{s_{E_1}(p),\;s_{E_2}(p),\;s_{E_1}(p)+s_{E_2}(p)\big\},\\
s_+(p)&:=\begin{cases}
\min\big\{s_{E_1}(p),\;s_{E_2}(p)\big\}
& \text{ if }  m(E_1\times E_2)=0, 
\\
\min\{s_{E_1}(p)+\frac{n_2}p,\;s_{E_2}(p)+\frac{n_1}p\big\}
& \text{ if } m(E_1\times E_2)>0.
\end{cases}
\end{align*}
Moreover, if either $p=2$ or $s_1,s_2\in\N_0$, and if $E_j$ are not $(s_{E_j}(p),p)$-null, $j=1,2$, then $E_1\times E_2$ is not $(s_-,p)$-null.
If $p\le2$, $m(E_1\times E_2)>0$, and $E_j$ are $(s_{E_j}(p),p)$-null, $j=1,2$, then $E_1\times E_2$ is $(s_+,p)$-null.
\end{prop}

We also mention Proposition \ref{prop:TensorNegative}, which states that tensor-product distributions cannot have higher Sobolev regularity than their factors.
In particular, we cannot directly use tensor-product distributions to prove that the Cartesian product $E_1\times E_2\subset\R^{n_1+n_2}$ is not $(s,p)$-null for any $s>\min\{n_1/p,n_2/p\}$.
From the results in \cite{HansenPhD} (discussed briefly in Appendix~\ref{s:TensorProduct}) we might conjecture that the upper bound in \eqref{eq:ProductBounds} in the case $m(E_1\times E_2)>0$ can be improved to $s_{E_1\times E_2}(p)\le  s_{E_1}(p) + s_{E_2}(p)$.

\begin{rem}
The bounds in \eqref{eq:ProductBounds} do not in general allow $s_{E_1\times E_2}(p)$ to be computed from $s_{E_1}(p)$ and $s_{E_2}(p)$ (unless $s_{E_1}(p)\cdot s_{E_2}(p)=0$ and $m(E_1\times E_2)=0$). 
That $s_{E_1}$ and $s_{E_2}$ do not in general determine $s_{E_1\times E_2}$ is shown by the following examples. 
If $E_1=E_2=\{0\}\subset\R$, then $s_{E_1}(p)=s_{E_2}(p)=-1/p'$ and $s_{E_1\times E_2}(p)=-2/p'=s_{E_1}(p)+s_{E_2}(p)=s_-(p)$, so the lower bound in  \eqref{eq:ProductBounds} is achieved. 
If $E_1,E_2\subset\R$ are Borel sets 
with Hausdorff dimension zero, for which $E_1\times E_2$ has Hausdorff dimension one (cf.\ e.g.\ Example~7.8 of \cite{Fal}), 
then by Theorem~\ref{thm:NullityHausdorff} below,
$s_{E_1\times E_2}(p)=-1/p'=s_{E_1}(p)=s_{E_2}(p)=s_+(p)$, so the upper bound in  \eqref{eq:ProductBounds} is achieved.
\end{rem}

We end this section with a simple application of $(s,p)$-nullity to function spaces on subsets of~$\R^n$.
\begin{prop}
\label{thm:Hs_equality_closed}
Let $1<p<\infty$, $s\in\R$, and let $F_1,F_2$ be closed subsets of $\R^n$. Then the following statements are equivalent:
\begin{enumerate}[(i)]
\item \label{a} The symmetric difference $F_1\ominus F_2$ is $(s,p)$-null.
\item \label{b} $F_1\setminus F_2$ and $ F_2\setminus F_1$ are both $(s,p)$-null.
\item \label{c}$H^{s,p}_{F_1}=H^{s,p}_{F_2}$.
\end{enumerate}
\end{prop}
\begin{proof}
That \rf{a} $\Leftrightarrow$ \rf{b} follows from Lemma \ref{lem:nullity1}\rf{aa} and Proposition \ref{thm:NullityOfUnions}\rf{ll}. 
To show that \rf{c} $\Rightarrow$ \rf{b} we argue by contrapositive. 
Suppose without loss of generality that $F_1\setminus F_2$ is not $(s,p)$-null. Then there exists $0\neq u\in H^{s,p}$ such that $\supp u \subset F_1\setminus F_2$, so that $u\in H^{s,p}_{F_1}$ but $u\not\in H^{s,p}_{F_2}$.
To show that \rf{b} $\Rightarrow$ \rf{c} we also argue by contrapositive. 
Suppose that $H^{s,p}_{F_1}\neq H^{s,p}_{F_2}$. Without loss of generality, we assume that there exists $u\in H^{s,p}_{F_1}\setminus H^{s,p}_{F_2}$. Let $\bx\in\supp{u}\cap(F_1\setminus F_2)$, and (by the closedness of $F_2$), let $\eps>0$ be such that $B_\eps(\bx)\cap F_2$ is empty. 
Then, for any $\phi\in \scrD(B_\eps(\bx))$ such that $\phi(\bx)\neq 0$, it holds 
(cf.\ the proof of Proposition~\ref{thm:NullityOfUnions}\rf{ll}) 
that $0\neq \phi u \in H^{s,p}$ with $\supp(\phi u)\subset F_1\setminus F_2$, which implies that $F_1\setminus F_2$ is not $(s,p)$-null.
\end{proof}

We now present our main theoretical results concerning $(s,p)$-nullity. 
Since Lemma \ref{lem:nullity1}\rf{qq} tells us that a set is $(0,p)$-null if and only if its inner Lebesgue measure is zero (independently of $p$), 
it makes sense to consider the cases $s<0$ and $s>0$ separately.

\subsection{The case \texorpdfstring{$s<0$}{s negative} (sets with zero measure)}
\label{subsec:ResultsSL0}

The following theorem provides a partial characterisation of $(s,p)$-nullity for $-n/p'\leq s<0$ in terms of Hausdorff dimension ${{\rm dim_H}}$ (defined, e.g.,\ in \cite[\S3]{Fal} or \cite[\S 5.1]{AdHe}).  
It will be proved at the end of \S\ref{sec:Capacity} using standard results from \cite[\S5]{AdHe} connecting Hausdorff dimension and capacity\footnote{We remark that the results in \cite[Chapter~5]{AdHe} actually allow a slightly more precise characterisation of $(s,p)$-nullity in terms of Hausdorff \emph{measure}. But we shall not pursue such characterisations here, since doing so would add considerable complexity with little gain in insight (indeed, \cite[\S5.6.4]{AdHe} implies that even Hausdorff measure is not sufficient to provide a \emph{complete} characterisation of $(s,p)$-nullity). In any case the results in Theorem \ref{thm:NullityHausdorff} seem sufficient for the applications of scattering by fractal screens \cite{CoercScreen,Ch:13,CoercScreen2,ScreenPaper} that motivate the current study.
}.
We note that Theorem \ref{thm:NullityHausdorff} applies as a special case to regular submanifolds of $\R^n$, and also to the ``multi-screens'', relevant for acoustic and electromagnetic scattering, considered e.g.\ in \cite{ClHi:13}. 
\begin{thm}
\label{thm:NullityHausdorff}
Let $1<p<\infty$ and $E\subset\R^n$.
\begin{enumerate}[(i)]
\item \label{hh} For $-n/p'<s\leq 0$, if $\dimH{E}< n+p's$, then $E$ is $(s,p)$-null.
\item \label{gg} For $-n/p'\leq s<0$, if $E$ is Borel and $(s,p)$-null, then $\dimH{E}\leq n+p's$.
\end{enumerate}
In particular, if $E$ is Borel and $m(E)=0$, then 
\begin{align}\label{eq:DimHCharac}
s_E(p) =\frac{\dim_H{E}-n}{p'}\qquad \text{and}\qquad
\dimH{E} = \inf\big\{d: E \mbox{ is } \big((d-n)/p',p\big)\mbox{-null}\big\}.
\end{align}
\end{thm}

\begin{rem}
\label{rem:TriebelHausdorff}
A link between $(s,p)$-nullity and fractal dimension was established previously in \cite{Li:67a}. Specifically, \cite[Theorem 4]{Li:67a} implies part \rf{hh} of Theorem \ref{thm:NullityHausdorff} for compact sets, with $\dimH$ replaced by $\underline{\dimB}$, the lower box (or Minkowski) dimension\footnote{The definition of lower box dimension in \cite{Li:67a} differs from the standard definition found e.g.\ in \cite[Equation (2.5)]{Fal}. The two definitions can be reconciled by noting that $\liminf_{r\to 0^+}-\frac{\log(N(r))}{\log (r)} = \inf\{ d\geq 0: \liminf_{r\to 0^+} N(r)r^d =0\}$ for any function $N:(0,\infty)\to [1,\infty)$.}. 
Since $\dimH(E)\leq \underline{\dimB}(E)$ for all bounded $E\subset \R^n$ 
\cite[Proposition 3.4]{Fal}, 
our result in part \rf{hh} is stronger than what is provided by \cite[Theorem~4]{Li:67a}.
Examples of sets for which $\dimH(E)< \underline{\dimB}(E)$ are easy to find: a particularly simple example is the set $E=\{0\}\cup\{1/n:n\in\N\}\subset\R$, for which $\dimH(E)=0$ but $\underline{\dimB}(E)=1/2$ (cf.\ \cite[Example 2.7]{Fal}).

Related results can also be found in \cite[Theorem 17.8]{Triebel97FracSpec}, where a formula similar to \eqref{eq:DimHCharac} is stated in the context of the spaces $B^s_{pq}$. 
However, Triebel's result concerns a different notion of nullity to ours---in place of the space $H^{s,p}_F$ in Definition \ref{def:nullity} he has the space
\begin{align*}
B^{s,F}_{pq}:=\{u\in B^s_{pq} : u(\psi) = 0 \mbox{ for all } \psi\in\scrS \mbox{ for which } \psi = 0 \mbox{ on }F\}.
\end{align*}
As Triebel points out in \cite[p.~126]{Triebel97FracSpec}, while $B^{s,F}_{pq}\subset \{u\in B^s_{pq}:\supp u\subset F\}$,  in general we do not have equality here. 
Since $B^s_{pq}\subset F^s_{p2}=H^{s,p}$ for $q\leq \min\{p,2\}$, Triebel's result implies part \rf{gg} of Theorem \ref{thm:NullityHausdorff}, but part \rf{hh} of Theorem \ref{thm:NullityHausdorff} is stronger than what is provided by Triebel's result.
\end{rem}

Theorem \ref{thm:NullityHausdorff} (specifically \rf{eq:DimHCharac}) provides a simple characterisation of the nullity threshold $s_E(p)$ for a Borel set $E$ with $\dimH{E}<n$. But it tells us nothing about ``threshold nullity'', i.e.\ whether or not $E$ is $(s_E(p),p)$-null. 
A general result concerning threshold nullity is given by the next proposition, which follows from Proposition~\ref{prop:AH551},  Theorem~\ref{thm:NullitycapEquiv}\rf{b2} and Remark \ref{rem:Compact}.
\begin{prop}\label{prop:ThresholdNullity}
Let $E\subset\R^n$ be Borel with $0\leq \dim_H E<n$, and let $1<q<p<\infty$ and $-\infty<s<t<0$ satisfy $tq'=sp'=\dim_H E-n$. If $E$ is $(s,p)$-null, then $E$ is $(t,q)$-null.
\end{prop}

The following corollary provides a full classification of all possible nullity and nullity threshold sets $\mathcal N_E$ and $\mathcal T_E$ (defined as in \rf{eqn:NullitySet}--\rf{eqn:ThresholdNullitySet}) that can arise when $m(E)=0$. It is a simple consequence of Proposition \ref{prop:ThresholdNullity}, Theorem~\ref{thm:NullityHausdorff} and Lemma~\ref{lem:nullity1}\rf{qq}. 
This result makes clear that the ``gap'' between parts \rf{hh} and \rf{gg} of Theorem \ref{thm:NullityHausdorff} cannot in general be bridged: 
no \emph{complete} characterisation of $(s,p)$-null sets for $-n/p'\leq s<0$ in terms of Hausdorff dimension is possible. 
The sharpness of our classification is demonstrated in Theorem \ref{thm:CantorZoo}, which provides the nullity and nullity threshold sets for a range of Cantor sets (defined in Definition \ref{def:Cantor}), for which the question of threshold nullity can be answered completely using Theorem \ref{thm:CantorCapacity}. 

\begin{cor}
\label{cor:NullitySets}
Let $E\subset \R^n$ be Borel with $m(E)=0$, and set $d=\dimH(E)\in[0,n]$. 
If $d=n$ then $\mathcal N_E=\big\{(s,p):s\ge0,\;1<p<\infty\big\}$ and hence $\mathcal T_E=(1,\infty)$. 
Otherwise, if $0\leq d<n$ then 
\begin{align}
&\big\{(s,p):s>(d-n)/p'\big\}\subset \mathcal N_E\subset \big\{(s,p):s\geq (d-n)/p'\big\},\notag\\ 
\textrm { and either } \quad &\mathcal T_E\in\Big\{\emptyset,\;(1,\infty)\Big\}, \;\textrm{ or } \;
\mathcal T_E \in\Big\{(1,\pp),\;(1,\pp]\Big\}, \textrm{ for some } 1<\pp<\infty.
\label{eq:PossibleNullitySets}
\end{align}
Moreover, Theorem~\ref{thm:CantorZoo} shows that this result cannot be improved: for every nullity set $\mathcal N\subset\R\times(1,\infty)$ allowed by \eqref{eq:PossibleNullitySets}, there exists a Cantor set $E^{(n)}\subset\R^n$ for which $\mathcal N_{E^{(n)}}=\mathcal N$.
\end{cor}

We now consider two other classes of sets for which it is possible to answer completely the question of threshold nullity. 

First, when $E$ consists of a single point, any distribution supported by $E$ is necessarily a linear combination of the delta function and its derivatives \cite[Theorem 3.9]{McLean}. In this case it follows from \eqref{eq:delta} that $s_E(p)=-n/p'$, and moreover that $E$ is $(-n/p',p)$-null. Proposition \ref{thm:NullityOfUnions} implies that the same holds for all countable sets. 
(We note however that countability is not a necessary condition for $(-n/p',p)$-nullity; a counterexample is provided by the Cantor set $F^{(n)}_{0,\infty}$ in Theorem~\ref{thm:CantorZoo}).
\begin{cor}
\label{cor:Countable}
A non-empty countable set is $(s,p)$-null if and only if $s\geq -n/p'$.
\end{cor}

Second, recall (e.g.\ \cite[\S3]{Triebel97FracSpec}) that for $0\leq d\leq n$ a closed set $F\subset \R^n$ with $\dimH(F)=d$ is called a $d$-set if there exist constants $c_1,c_2>0$ such that 
\begin{align}
\label{eq:dSet}
0<c_1 r^d \leq \mathcal H^d(B_r(\bx)\cap F) \leq c_2 r^d<\infty, \qquad \textrm{for all } \bx\in F, \; 0<r<1,
\end{align}
where $\mathcal H^d$ is the $d$-dimensional Hausdorff measure on $\R^n$. (Note that this definition differs from that used in the fractal geometry literature, e.g.\ \cite[p.~48]{Fal}.)
Condition \eqref{eq:dSet} may be understood as saying that $d$-sets are everywhere locally $d$-dimensional.
Note that the definition of $d$-set includes as a special case all Lipschitz $d$-dimensional manifolds, $d\in\{0,1,\ldots, n\}$ (cf.\ also \ref{thm:Domains}\rf{kk} below). 
By combining Theorem \ref{thm:NullityDensityEquiv} with results due to Triebel on the density of test functions in function spaces \cite[Theorems 3 and 5]{Tri:08} one can prove the following result.
\begin{thm}
\label{thm:dSet}
Let $1<p<\infty$, and $0<d<n$. Let $F\subset \R^n$ be either a compact $d$-set, or a $d$-dimensional hyperplane (in which case $d$ is assumed to be an integer). Then $F$ is $((d-n)/p',p)$-null.
\end{thm}

Our final theorem in this section applies the results of Theorem \ref{thm:NullityHausdorff} to the special case where the set $E$ is the boundary of an open set $\Omega\subset\R^n$. Its proof makes use of Lemma \ref{lem:domains-dimH} in \S\ref{sec:Domains}, which collects a number of results concerning the relationship between the analytical regularity of the boundary of a set and its fractal dimension. The proof of part \rf{kk} of the theorem is postponed until~\S\ref{sec:Capacity}. 

Here and in what follows we shall say that a non-empty open set $\Omega$ is $C^0$ (respectively $C^{0,\alpha}$, $0<\alpha<1$, respectively Lipschitz) if its boundary $\partial\Omega$ can be locally represented as the graph (suitably rotated) of a $C^0$ (respectively $C^{0,\alpha}$, respectively Lipschitz) function from $\R^{n-1}$ to $\R$, with $\Omega$ lying only on one side of $\partial\Omega$. 
For a more precise definition see \cite[1.2.1.1]{Gri}.
We note that for $n=1$ there is no distinction between these definitions: we interpret them all to mean that $\Omega$ is a countable union of open intervals whose closures are disjoint. We also point out that in the literature several alternative definitions of Lipschitz open sets can be found (for a detailed discussion see e.g.\ \cite{Fr:79,Gri}); in particular, our definition includes Stein's ``minimally smooth domains'' \cite[{\S}VI.3.3]{Stein}.
\begin{thm}
\label{thm:Domains}
Let $1<p<\infty$ and let $\Omega\subset\R^n$ be non-empty and open.
\begin{enumerate}[(i)]
\item \label{kk1} 
If $\Omega^c$ has non-empty interior then $\deO$ is not $(s,p)$-null for $s<-1/p'$. (In particular this holds if $\Omega\neq \R^n$ is $C^0$.)
\item \label{kk0} 
If $\Omega$ is $C^0$ and $s\geq 0$, then $\partial\Omega$ is $(s,p)$-null. 
\item \label{kk2} 
If $\Omega$ is $C^{0,\alpha}$ for some $0<\alpha<1$ and $s> -\alpha/p'$, then $\partial\Omega$ is $(s,p)$-null. 
\item \label{kk} If $\Omega$ is Lipschitz then $\partial\Omega$ is $(s,p)$-null if and only if $s\geq -1/p'$.
\end{enumerate}
\end{thm}

\begin{proof}
The case $n=1$ is covered by Corollary \ref{cor:Countable}, so assume $n\geq 2$. 
For \rf{kk1}, Lemma \ref{lem:domains-dimH}\rf{app.a} states that $\dimH{\partial\Omega}\geq n-1$ and then Theorem \ref{thm:NullityHausdorff}\rf{gg} implies that $\deO$ is not $(s,p)$-null for any $s<-1/p'$.
\rf{kk0} follows from Lemma \ref{lem:nullity1}\rf{qq}
 and Lemma \ref{lem:domains-dimH}\rf{app.aa}. 
\rf{kk2} follows from Theorem \ref{thm:NullityHausdorff}\rf{hh} and Lemma \ref{lem:domains-dimH}\rf{app.b}.
\rf{kk} is proved in \S\ref{sec:Capacity}. 
\end{proof}

In \S\ref{sec:Domains} we provide concrete examples to demonstrate the sharpness of these results. In particular, Lemma \ref{lem:domains-dimH} implies that for $n\geq2$ there exists a bounded $C^{0,\alpha}$ open set whose boundary is not $(s,p)$-null for any $s<-\alpha/p'$, and a bounded $C^0$ open set whose boundary is not $(s,p)$-null for any $s<0$.

\subsection{The case \texorpdfstring{$s>0$}{s positive} (sets with non-zero measure)}
\label{subsec:ResultsSG0}

As has been discussed above, questions of $(s,p)$-nullity for $s<0$ can often be answered by appealing to Theorem \ref{thm:NullityCapEquiv} and applying standard potential theoretic results on the capacity ${\rm Cap}_{t,p'}$ with $t=-s>0$. When it comes to investigating $(s,p)$-nullity for $s>0$, however, Theorem \ref{thm:NullityCapEquiv} appears to be of little use because the properties of ${\rm Cap}_{t,p'}$ for $t<0$ do not seem to have been widely documented. Certainly, if a set $E$ is to have nullity threshold in $(0,n/p]$ it must have non-zero measure and empty interior (cf.\ Lemma \ref{lem:nullity1}\rf{cc},\rf{qq}). But we are not aware of any general characterisations of $(s,p)$-nullity for $s\in(0,n/p]$ in terms of the geometrical properties of a set, analogous to that provided by Hausdorff dimension for nullity thresholds in $[-n/p',0]$. 

The following theorem provides an alternative analytic characterisation in terms of the capacity ${\rm cap}_{s,p}$ defined in \S\ref{sec:Capacity}.  
It is taken from \cite[Theorem 11.3.2]{AdHe}, where it is stated with part \rf{a3} replaced by an equivalent statement in terms of sets of uniqueness (cf.\ \S\ref{subsec:SOU}) and the assumption that $F$ be closed relaxed to $F$ being Borel. 
It generalises an earlier result presented in \cite[Theorem 2.6]{Po:72}.
\begin{thm}[{\cite[Theorem 11.3.2]{AdHe}}]
\label{thm:SOU}
Let $1<p<\infty$ and $0<s\leq n/p$. Let $F$ be closed with empty interior. Then the following are equivalent:
\begin{enumerate}[(i)]
\item \label{a3} $F$ is $(s,p)$-null;
\item \label{b3} $\ccap_{s,p}(\Omega\setminus F)=\ccap_{s,p}(\Omega)$ for all open $\Omega\subset \R^n$;
\item \label{c3} $\ccap_{s,p}(B_\delta(\bx)\setminus F)=\ccap_{s,p}(B_\delta(\bx))$ for all open balls $B_\delta(\bx)\subset \R^n$;
\item \label{d3} For almost all $\bx\in\R^n$ (with respect to Lebesgue measure)
\[ \limsup_{\delta\to 0} \frac{\ccap_{s,p}(B_\delta(\bx)\setminus F)}{\delta^n}>0.\]
\end{enumerate}
\end{thm}

As is pointed out in \cite[p.~314]{AdHe}, given $1<p<\infty$ and $s\in (0,n/p]$, this characterisation allows us to construct compact sets $K\subset\R^n$ with positive measure and empty interior which are not $(s,p)$-null, by engineering the failure of condition \rf{b3} above. The approach suggested in \cite[p.~314]{AdHe} (described in more detail in \S\ref{subsec:SwissCheese} below), is based on a standard ``Swiss cheese'' construction. 
One starts with a bounded open set $\Omega\subset\R$ and removes from $\overline{\Omega}$ a countable sequence of open balls of diminishing radius in such a way that the remaining compact set $K\subset\overline{\Omega}$ has empty interior and satisfies $\ccap_{s,p}(\Omega\setminus K)<\ccap_{s,p}(\Omega)$. As we will explain in \S\ref{subsec:Examplessgreaterthanzero} (see in particular Theorem \ref{thm:SwissCheeseNullity}), sufficient conditions to ensure the latter bound can be obtained using the countable sub-additivity of capacity (Proposition \ref{thm:CountableSubadd}) and standard estimates on the capacity of balls (Proposition \ref{prop:CapBall}).

In \S\ref{subsec:Examplessgreaterthanzero} we apply this methodology to derive sufficient conditions for the non-$(s,p)$-nullity of certain fat Cantor sets (defined in Examples \ref{ex:FatCantor} and \ref{ex:SuperFatCantor}). In \S\ref{subsec:Examplessgreaterthanzero} we also prove similar but complementary results using a completely different methodology not involving capacity, adapted from \cite{La:11}, which is based on direct estimates of the Sobolev norm of the characteristic function of the set for the case $p=2$, obtained via explicit bounds on its Fourier transform.

However, these two approaches (i.e.,\ proving upper bounds on capacities, or on Sobolev norms) do not in general allow us to calculate the nullity threshold $s_K(p)$ for the compact set $K$ under consideration; they only provide a lower bound on $s_K(p)$ (by proving the existence of some $\tilde s\in(0,n/p]$ for which $K$ is not $(\tilde s,p)$-null). Since $K$ is assumed to have empty interior, all we can deduce is that $s_K\in [\tilde s,n/p]$. 
Only in the extreme case $\tilde s=n/p$ does such a non-nullity result specify the nullity threshold exactly. 
The existence of a compact set $K_p$ with empty interior which is not $(n/p,p)$-null appears to have been proved first by Polking in \cite[Theorem 4]{Po:72a}. Polking's set $K_p$ is a ``Swiss cheese'' set of the kind described above, but Polking's analysis does not make use of the capacity-theoretic characterisation of Theorem \ref{thm:SOU}; instead Polking provides an explicit construction of a non-zero function $f_p\in H^{n/p,p}_{K_p}$, 
appealing to Leibniz-type formulae for fractional derivatives in order to prove that $\|f_p\|_{H^{n/p,p}}<\infty$. 
As Polking remarks in \cite{Po:72a}, this result ``\emph{illustrates in a rather striking manner that the Sobolev embedding theorem is sharp}''.

\begin{thm}[{\cite[Theorem 4]{Po:72a},\cite[p.~314]{AdHe}}]
\label{thm:PolkingCheese}
Let $1<p<\infty$. There exists a compact set $K_p\subset\R^n$ with empty interior which is not $(n/p,p)$-null. 
In particular, $s_{K_p}(p)=n/p$.
\end{thm}

The set $K_p$ whose existence is guaranteed by Theorem \ref{thm:PolkingCheese} is, at least \emph{a priori}, $p$-dependent. (Certainly the constructions in \cite[Theorem 4]{Po:72a} and \cite[p.~314]{AdHe} are intrinsically $p$-dependent.) By taking the closure of the union of a countable sequence of such sets $K_p$, suitably scaled and translated, one can construct a compact set with empty interior that is not $(n/p,p)$-null for any $1<p<\infty$.
\begin{cor}
\label{cor:Compact}
There exists a compact set $K\subset\R^n$ with empty interior which is not $(n/p,p)$-null for any $1<p<\infty$. In particular, $s_K(p)=n/p$ for every $1<p<\infty$.
\end{cor}
\begin{proof}
For each $j\in \N$ let $p_j= 1+1/j$, $\bx_j=(2^{1-j},0,\ldots,0)$ and let $K_{p_j}\subset [0,1]^n$ be a compact set with empty interior which is not $(n/p_j,p_j)$-null. Then $K=\{\mathbf{0}\}\cup \bigcup_{j=1}^\infty (\bx_j+2^{-j}K_{p_j})$ is compact, has empty interior, and is not $(n/p,p)$-null for any $1<p<\infty$.
\end{proof}

The following related result is a by-product of the arguments leading to Theorem \ref{thm:PolkingCheese}, and gives the nullity threshold of certain non-compact sets, for example $\R^n\setminus\Q^n$.
\begin{thm}\label{thm:OpenMinusDense}
Let $A\subset\R^n$ have non-empty interior and $Q\subset A$ be countable and dense in the interior of $A$.
Then $E:=A\setminus Q$ is not $(n/p,p)$-null, and hence $s_E(p)=n/p$, for all $1<p<\infty$.
\end{thm}
\begin{proof}
Given $A$ and $Q$ as in the assertion and $1<p<\infty$, the Swiss-cheese construction of \cite[Theorem 4]{Po:72a}, applied to any non-empty bounded open subset of $A$, gives a set $K_p\subset A\setminus Q$ that is not $(n/p,p)$-null.
Then $A\setminus Q$ is not $(n/p,p)$-null for any $p$ and has empty interior, so the theorem follows.
\end{proof}

Proving the existence of sets whose nullity threshold lies in the open interval $(0,n/p)$ appears to be an \textbf{open problem}---certainly we are not aware of any literature on this matter. The difficulty here is that one would need to prove that a set with positive measure is $(s,p)$-null for some $s\in(0,n/p)$. The possibility of doing this for a closed set $F$ using the capacity-theoretic characterisation of Theorem \ref{thm:SOU} seems remote. Using the conditions \rf{b3} or \rf{c3} from that theorem would require us to prove \emph{equality} of two capacities, which is difficult because capacity can usually only be estimated rather than computed exactly. Condition \rf{d3} could in principle be verified by proving a sufficiently sharp lower bound on $\ccap_{s,p}(B_\delta(\bx)\setminus F)$, but even this seems difficult in general as lower bounds for $\ccap_{s,p}$ appear to be available only for balls (by contrast, quite general upper bounds can be obtained using countable sub-additivity, as we have 
mentioned above). 
We also remark that the Fourier approach adopted in Proposition~\ref{prop:FatCantor} is not promising in this regard: lower bounds for $|\hat u|$ might be used to show that a particular $u$ is not in a certain $\Hsp_F$, but this would not rule out the existence of other non-trivial functions in $\Hsp_F$.

\section{Capacity}
\label{sec:Capacity}
In this section we provide proofs of Proposition \ref{thm:NullityOfUnions}\rf{mm}, Theorem \ref{thm:NullityHausdorff}, and Theorem \ref{thm:Domains}\rf{kk}. Our arguments rely on Theorem \ref{thm:NullityCapEquiv}, which characterises $s$-null sets in terms of a set function called \emph{capacity}. The notion of capacity is central to potential theory, and we briefly review some of the basic ideas here. Our presentation is based broadly on \cite{AdHe,Maz'ya}, but other relevant references include \cite{Ca:67,Li:67b,Me:70,Po:72,AdPo:73,Ad:74,Ziemer,Ch:54}. 
We begin with a rather general definition of capacity before specialising to the particular capacities of relevance to the problem in hand. Since (i) the literature is in places highly technical; (ii) notational conventions are varied and sometimes conflicting (see the discussion 
in Remark~\ref{rem:SvsD} below); and (iii) the concept of capacity may not be familiar to some readers, we take care to clarify certain details that are not fully explained in \cite{AdHe,Maz'ya}. 
\begin{defn}
\label{def:cap}
Let $\Ccomp$ be a set function defined on compact subsets of $\R^n$, taking values in $[0,\infty]$, such that $\Ccomp(\emptyset)=0$ and $\Ccomp(K_1)\leq \Ccomp(K_2)$ for all compact $K_1\subset K_2 \subset \R^n$. 
From $\Ccomp$ we define \emph{inner} and \emph{outer capacities} on arbitrary subsets $E\subset\R^n$ by
\begin{align*}
\underline{C}(E):=\sup_{\substack{K\subset E\\ K \textrm{ compact}}} \Ccomp(K),
\qquad\qquad 
\overline{C}(E):=\inf_{\substack{\Omega\supset E\\ \Omega \textrm{ open}}}  \underline{C}(\Omega).
\end{align*}
Clearly $\underline{C}(E)\leq \overline{C}(E)$ for all $E\subset\R^n$. If $\underline{C}(E)=\overline{C}(E)$ then we say $E$ is \emph{capacitable} for $\Ccomp$ and define the \emph{capacity} of $E$ to be $C(E):=\underline{C}(E)=\overline{C}(E)$.
\end{defn}
It follows straight from the definitions that open sets are capacitable, and that $\underline{C}(K)= \Ccomp(K)$ for all compact $K\subset\R^n$. It is common practice \cite{AdHe,Maz'ya} to denote the original set function from which $\uC$ and $\oC(E)$ are defined simply by $C$, rather than $\Ccomp$. No ambiguity arises from this abuse of notation provided that compact sets are capacitable; that this is the case for the capacities of interest to us will be demonstrated in Proposition \ref{prop:CompactCapacitable} below. 

We now define two particular capacities of relevance for the study of $(s,p)$-nullity. 
For $1<p<\infty$, $s\in\R$ and $K\subset\R^n$ compact, we define
\begin{align}
\label{capdefn}
\capcomp_{s,p}(K)&:= \inf  \{\|u \|^p_{H^{s,p}}:u\in \scrD \textrm{ and } u\geq 1 \textrm{ in a neighbourhood of } K \},\\
\label{Capdefn}
\Capcomp_{s,p}(K)&:= \inf  \{\|u \|^p_{H^{s,p}}:u\in \scrD \textrm{ and } u=1 \textrm{ in a neighbourhood of } K \}.
\end{align}
Clearly $\capcomp_{s,p}(K)\leq \Capcomp_{s,p}(K)$ for all compact sets $K$. It is a much deeper fact that, at least for $s>0$, the reverse inequality also holds, up to a constant factor---see Theorem \ref{thm:CapEquiv} below.

\begin{rem}
\label{rem:SvsD}
The capacities $\ccap_{s,p}$ and $\cCap_{s,p}$ arising from \rf{capdefn} and \rf{Capdefn} are classical and appear throughout the potential theory literature.  
Our notation is adapted from \cite[\S10.4 and \S13.1]{Maz'ya}, where $\ccap_{s,p}(\cdot)$ and $\cCap_{s,p}(\cdot)$ are respectively denoted $\ccap(\cdot,H^{s}_p(\R^n))$ and $\cCap(\cdot,H^{s}_p(\R^n))$. But many other conflicting notational conventions exist: for instance, 
$\{\capcomp_{s,p},\Capcomp_{s,p}\}$ are respectively denoted 
$\{C_{s,p},N_{s,p}\}$ in \cite[\S2.2, \S2.7]{AdHe}, 
$\{N_{s,p},M_{s,p}\}$ in \cite{Li:67b},
and $\{B_{s,p},C_{s,p}\}$ in \cite{AdPo:73,Po:72}; $\capcomp_{s,p}$ is denoted $B_{s;p}$ in \cite{Me:70}, and $B^{(n)}_{s,p}$ in \cite{Ad:74}. 

Navigating the literature is also complicated by the fact that $\ccap_{s,p}$ and $\cCap_{s,p}$ can be defined in a number of equivalent ways. Firstly, Definitions \rf{capdefn}--\rf{Capdefn} are sometimes stated with the trial functions $u$ ranging over $\scrS$ (as in \cite[pp.~19--20]{AdHe}) rather than $\scrD$ (as in \cite[\S13.1]{Maz'ya}). That these two choices of trial space lead to the same set functions (and hence the same capacities) is straightforward to prove.
Let $\tcapcomp_{s,p}$ denote the set function defined by \rf{capdefn} using $\scrS$ instead of $\scrD$, and let $K\subset\R^n$ be compact.
Obviously $\tcapcomp_{s,p}(K)\leq \capcomp_{s,p}(K)$; 
for a bound in the opposite direction, consider any $u\in\scrS$ with $u\geq 1$ on an open neighbourhood $\Omega$ of $K$. 
Take $R>0$ such that $\overline\Omega\subset B_R(\mathbf{0})$, and take a cutoff $\chi\in\scrD$ such that $\chi=1$ in $B_R(\mathbf{0})$.
Then $w:=(1-\chi)u\in\scrS$, with support in the complement $O:=\R^n\setminus \overline{B_R(\mathbf{0})}$, so that given $\eps>0$ there exists $\psi_\eps\in\scrD(O)$ such that $\|w-\psi_\eps\|_{H^{s,p}}<\eps$.
This bound implies that $\eta_\eps:=\chi u+\psi_\eps\in \scrD$ satisfies
$\|u-\eta_\eps\|_{H^{s,p}}<\eps$, so that $\|\eta_\eps\|_{H^{s,p}}< \|u\|_{H^{s,p}}+\eps$; note also that $\eta_\eps=\chi u+0 \geq 1$ on $\Omega$. 
Since $u$ and $\eps>0$ were arbitrary we conclude that $\tcapcomp_{s,p}(K)\geq \capcomp_{s,p}(K)$, and hence that $\tcapcomp_{s,p}(K)= \capcomp_{s,p}(K)$, as claimed. 
The analogous result for $\cCap_{s,p}$ follows by a similar argument with ``$\geq 1$'' replaced by ``$=1$'' throughout.

Secondly, Definition \rf{capdefn} is sometimes stated (e.g.\ \cite[\S2.2]{AdHe} and \cite[\S13.1]{Maz'ya}) with ``on $K$'' instead of ``in a neighbourhood of $K$''. Again it is easy to verify that the two definitions are equivalent. Let $\hcapcomp_{s,p}$ denote the set function defined by \rf{capdefn} using ``on $K$'' instead of ``in a neighbourhood of $K$''. Then clearly $\hcapcomp_{s,p}(K)\leq \capcomp_{s,p}(K)$; for a bound in the opposite direction, note that, given $\alpha\in (0,1)$, if $u\geq 1$ on $K$ then there exists a neighbourhood of $K$ on which $u\geq \alpha$. 
Hence 
$\capcomp_{s,p}(K)\leq \alpha^{-p}\hcapcomp_{s,p}(K)$, 
and since this holds for $\alpha$ arbitrarily close to $1$, we conclude that $\hcapcomp_{s,p}(K)= \capcomp_{s,p}(K)$, as claimed.

Finally, for $s>0$ some authors use a definition of $\capcomp_{s,p}$ in which the right-hand-side of \rf{capdefn} is replaced by an infimum of $\|f\|_{L^p}^p$ over the non-negative $f\in L^p$  for which $\mathcal{J}_{-s}f \geq 1$ on $K$ (cf.\ e.g.\ \cite[Definition 2.1]{Po:72}). That this definition is equivalent to \rf{capdefn} is proved in \cite[Proposition 2.3.13]{AdHe}.
\end{rem}

\begin{rem}
\label{rem:CapMeasure}
For $s=0$ one can show using standard measure-theoretic techniques that the capacities $\ccap_{0,p}$ and $\cCap_{0,p}$ both coincide with the Lebesgue measure on $\R^n$. Specifically, for $E\subset\R^n$ let $\underline{m}(E)=\sup\{m(K):E\supset K,\,K \textrm{ compact}\}$ and $\overline{m}(E)=\inf\{m(\Omega):E\subset \Omega,\,\Omega \textrm{ open}\}$ be the usual inner and outer Lebesgue measures of $E$ \cite[Definitions~2.2--2.3]{BBT97}.
Then for every $1<p<\infty$ it holds that 
$\ucap_{0,p}(E) = \uCap_{0,p}(E) = \underline{m}(E)$ and $\ocap_{0,p}(E) = \oCap_{0,p}(E) = \overline{m}(E)$. 
Hence $E$ is capacitable for $\capcomp_{0,p}$ (equivalently for $\Capcomp_{0,p}$) if and only if $E$ is Lebesgue measurable, in which case $\ccap_{0,p}(E)=\cCap_{0,p}(E)=m(E)$, where $m$ is the Lebesgue measure.
\end{rem}

As promised, we now prove the capacitability of compact sets for $\capcomp_{s,p}$ and $\Capcomp_{s,p}$.
\begin{prop}
\label{prop:CompactCapacitable}
Compact sets are capacitable for both $\capcomp_{s,p}$ and $\Capcomp_{s,p}$, for all $1<p<\infty$ and $s\in\R$.
\end{prop}
\begin{proof}
The result for $\capcomp_{s,p}$ is stated and proved in \cite[Proposition~2.2.3]{AdHe} for integer $s>0$, but the same proof is in fact valid for all $s\in\R$. 
To prove the result for $\Capcomp_{s,p}$, let $K\subset\R^n$ be compact and, given $\eps>0$, let $u\in\scrD$ satisfy $u=1$ in a neighbourhood $\Omega$ of $K$, with $\|u\|_{H^{s,p}}^p<\Capcomp_{s,p}(K)+\eps$. 
Hence $\Capcomp_{s,p}(\tilde{K})\leq \|u\|_{H^{s,p}}^p<\Capcomp_{s,p}(K)+\eps$ for all compact $\tilde{K}\subset \Omega$, which implies that 
$\cCap_{s,p}(\Omega)=\uCap_{s,p}(K)\leq \oCap_{s,p}(K)\leq\cCap_{s,p}(\Omega)\leq \Capcomp_{s,p}(K)+\eps$. Since $\eps$ was arbitrary, 
we conclude that $K$ is capacitable for $\Capcomp_{s,p}$.
\end{proof}

The link between capacity and nullity was stated in Theorem \ref{thm:NullityCapEquiv} above: to repeat, a set $E$ is $(s,p)$-null if and only if $\uCap_{-s,p'}(E)=0$. 
However, relatively little seems to be known about the capacity $\cCap_{s,p}$; the capacity $\ccap_{s,p}$ appears to be much better understood. 
In particular, there is a class of capacities, known as ``Choquet capacities'' (cf.\ \cite[Theorem 2.3.11]{AdHe} and Choquet's original work \cite{Ch:54}) for which all Suslin sets (defined in \cite[\S2.9, Notes to \S2.3]{AdHe}), in particular all Borel sets, are capacitable. The capacity $\ccap_{s,p}$ is well-known to be of this class \cite[\S2.3]{AdHe} (and see also \cite{Me:70}), at least for $s\geq 0$. 
But it has been suggested that the same is probably not true of $\cCap_{s,p}$ \cite[p.~1236]{Po:72} (although of course it is true for $\cCap_{0,p}$, cf.\ Remark \ref{rem:CapMeasure}).
\begin{prop}[{\cite[Propositions 2.3.12 and 2.3.13]{AdHe}}]
\label{thm:BorelCap}
Borel sets are capacitable for $\capcomp_{s,p}$ for $1<p<\infty$ and $s\geq 0$.
\end{prop}

\begin{rem}\label{rem:Borel}
In Theorem~\ref{thm:NullitycapEquiv}\rf{b2}, 
Proposition~\ref{thm:NullityOfUnions}\rf{mm},
Proposition~\ref{cor:TensorProduct},
Theorem~\ref{thm:NullityHausdorff}\rf{gg}, 
Proposition~\ref{prop:ThresholdNullity}, 
Theorem~\ref{thm:CantorZoo} 
and Proposition~\ref{prop:Trace}
we require certain sets to be Borel.
This assumption is made solely to allow application of Proposition~\ref{thm:BorelCap}. Hence, if desired, throughout the paper ``Borel'' may be substituted by ``Suslin'', or possibly by a more general class of capacitable sets.
\end{rem}

The outer capacity $\ocap_{s,p}$ is also known to be countably subadditive for $s\geq 0$ \cite[\S2.3]{AdHe}. The authors are not aware of a similar result for $\cCap_{s,p}$, but the example in Remark \ref{rem:ExampleRQ}, together with Theorem \ref{thm:NullityCapEquiv}, shows that $\uCap_{s,p}$ is not subadditive (not even finitely) for $s<-n/p'$.
\begin{prop}[{\cite[Propositions 2.3.6 and 2.3.13]{AdHe}}]
\label{thm:CountableSubadd}
Let $1<p<\infty$, $s\geq 0$ and let $E_i\subset \R^n$, $i\in\N$. Then
\begin{align*}
\ocap_{s,p}\left(\bigcup_{i=1}^\infty E_i\right) \leq \sum_{i=1}^\infty \ocap_{s,p}(E_i).
\end{align*}
\end{prop}

The link between the analytical concept of capacity and the geometrical concept of fractal dimension is provided by the following theorem, which provides a partial characterisation of the sets of zero outer capacity ${\rm \overline{cap}}_{s,p}(E)$ for $0<s\leq n/p$ in terms of Hausdorff dimension. The theorem, which we state without proof, is essentially a rephrasing of the results in \cite[\S5.1]{AdHe} (specifically Theorems 5.1.9 and 5.1.13). Similar results can be found e.g.\ in \cite[\S10.4.3]{Maz'ya}, \cite[\S8]{Me:70} and \cite[Theorem 2.6.16]{Ziemer}. For a historical background to these results the reader is referred to \cite[\S5.7]{AdHe}.
\begin{thm}[{\cite[Theorems 5.1.9 and 5.1.13]{AdHe}}]
\label{thm:CapDimH}
Let $1<p<\infty$ and $E\subset\R^n$. 
\begin{enumerate}[(i)]
\item \label{b1} For $0\leq s<n/p$, if ${\rm dim_H}(E)< n-ps$ then ${\rm \overline{cap}}_{s,p}(E)=0$.
\item \label{a1} For $0 \le s\leq n/p$, if ${\rm \overline{cap}}_{s,p}(E)=0$ then ${\rm dim_H}(E)\leq n-ps$.
\end{enumerate}
\end{thm}

The behaviour of $\ocap_{s,p}$ under Lipschitz mappings is also understood 
\cite[\S5.2]{AdHe}.
\begin{thm}[{\cite[Theorem 5.2.1]{AdHe}}]
\label{thm:Lipschitz}
Let $E\subset\R^n$, and let $\Phi:E\to\R^n$ be a Lipschitz map with Lipschitz constant $L$. Then for $1<p<\infty$ and $0\leq s\leq n/p$ there exists a constant $a>0$, depending only on $n,p,s$ and $L$, such that
\begin{align*}
\ocap_{s,p}\left(\Phi(E)\right) \leq a\,\ocap_{s,p}(E).
\end{align*}
\end{thm}

A further useful result on $\ocap_{s,p}$ is the following.
\begin{prop}[{\cite[Theorem 5.5.1]{AdHe}}]\label{prop:AH551}
Let $E\subset\R^n$ be bounded, and let $s,t\in\R$ and $1<p,q<\infty$ be such that either $0<tq<sp\le n$ or $p<q$ and $0<tq=sp\le n$.
Then, $\ocap_{s,p}(E)=0$ implies that $\ocap_{t,q}(E)=0$.
\end{prop}
Note that \cite[Theorem 5.5.1]{AdHe} requires $E$ to have diameter at most one, as that theorem deals with the actual values of the capacities.
Since here we are only concerned with the vanishing of the same capacities, by affine scaling the result holds for any bounded set.

To prove Proposition \ref{thm:NullityOfUnions}\rf{mm}, Theorem \ref{thm:NullityHausdorff}, and Theorem \ref{thm:Domains}\rf{kk}, which was the goal of this section, we have to link the concept of nullity (which by Theorem \ref{thm:NullityCapEquiv} concerns $\cCap_{s,p}$) with the results of Proposition \ref{thm:BorelCap} and Theorem \ref{thm:Lipschitz} (which concern $\ccap_{s,p}$). The link, as was hinted at just before Remark \ref{rem:SvsD}, is that the two capacities $\ccap_{s,p}$ and $\cCap_{s,p}$ are equivalent, at least for $s\geq 0$, in the sense of the following definition. 
\begin{defn}
\label{def:CapEquiv}
Let $\Ccomp$ and $\tCcomp$ be set functions satisfying Definition \ref{def:cap}. The resulting capacities $C$ and $\tilde C$ are said to be \emph{equivalent} if there exist constants $a,b>0$ such that
\begin{align*}
a \Ccomp(K) \leq \tCcomp(K) \leq b\Ccomp (K), \qquad \textrm{for all compact }K\subset \R^n.
\end{align*}
\end{defn}

\begin{prop}
\label{prop:CapEquiv}
If two capacities $C$ and $\tilde C$ are equivalent then, for any $E\subset \R^n$,
\begin{align*}
a \uC(E) \leq \utC(E) \leq b\uC (E), \qquad
a \oC(E) \leq \otC(E) \leq b\oC (E),
\end{align*}
where $a,b$ are the constants in Definition \ref{def:CapEquiv}. 
In particular, $E$ is capacitable for $\Ccomp$ with $C(E)=0$ if and only if $E$ is capacitable for $\tCcomp$ with $\tilde{C}(E)=0$.
\end{prop}

\begin{thm}[{\cite[eq. (2.7.4), Corollary~3.3.4, and Notes to \S2.9 and \S3.8]{AdHe}}] 
\label{thm:CapEquiv}
For every $1<p<\infty$ and $s\geq 0$ the capacities $\ccap_{s,p}$ and $\cCap_{s,p}$ are equivalent. 
Specifically, for any $s\geq 0$ there exists $b\geq 1$ such that, for all compact $K$,
\begin{align*}
\capcomp_{s,p}(K)\leq \Capcomp_{s,p}(K)\leq b\, \capcomp_{s,p}(K).
\end{align*}
\end{thm}

\begin{rem}
It is noted in \cite[Notes to \S2.7]{AdHe}) that results due to Deny \cite[Th\'eor\`eme II:3, p.~144]{De:50} imply that for $p=2$ and $0<s\leq 1$, the constant $b$ in \rf{thm:CapEquiv} can be taken to be $1$, so that $\ccap_{s,p}$ and $\cCap_{s,p}$ coincide (we have already noted this result for $s=0$ in Remark \ref{rem:CapMeasure}). An interesting \textbf{open question} concerns the extent to which this result generalises to $p\neq 2$ and/or $s\not\in [0,1]$. In Appendix \ref{app:CapacityEquivalence} we demonstrate that the result is certainly not true for $p=2$ and $s=2$, using an explicit formula for the norm in the restriction space $H^{2,2}(\Omega)$ recently presented in \cite{InterpolationCWHM}.
\end{rem}

Theorem \ref{thm:NullityCapEquiv}, Proposition \ref{thm:BorelCap}, Proposition \ref{prop:CapEquiv} and Theorem \ref{thm:CapEquiv} then provide the following key result, which allows us to complete the proofs of the remaining results stated in \S\ref{sec:MainResults}.
\begin{thm}
\label{thm:NullitycapEquiv}
Let $1<p<\infty$, $s\leq 0$ and $E\subset\R^n$. 
\begin{enumerate}[(i)]
\item \label{a2} If $\ocap_{-s,p'}(E)=0$ then $E$ is $(s,p)$-null.
\item \label{b2} If $E$ is Borel, then $\ocap_{-s,p'}(E)=\ccap_{-s,p'}(E)=0$ if and only if $E$ is $(s,p)$-null.
\end{enumerate}
\end{thm}

\begin{proof}[Proof of Proposition \ref{thm:NullityOfUnions}\rf{mm}]
This follows immediately from Theorem \ref{thm:NullitycapEquiv} and Proposition~\ref{thm:CountableSubadd}.
\end{proof}

\begin{proof}[Proof of Theorem \ref{thm:NullityHausdorff}]
Part \rf{hh} follows from Theorem \ref{thm:NullitycapEquiv}\rf{a2} and Theorem \ref{thm:CapDimH}\rf{b1}. 
Part \rf{gg} follows from Theorem \ref{thm:NullitycapEquiv}\rf{b2} and Theorem \ref{thm:CapDimH}\rf{a1}. 
\end{proof}

\begin{proof}[Proof of Theorem \ref{thm:Domains}\rf{kk}]
By applying a suitable smooth cutoff and a coordinate rotation, it suffices to consider the case where $\Omega=\{\bx\in\R^n:x_n <\phi(x_1,\ldots,x_{n-1})\}$, where $\phi:\R^{n-1}\to \R$ is Lipschitz. Defining $\R^n_0:=\{\bx\in\R^n:x_n=0\}$, the map $\Phi:\R^n_0\subset\R^n\to\partial\Omega\subset\R^n$ is Lipschitz with a Lipschitz inverse (given by the orthogonal projection of $\partial\Omega$ onto $\R^n_0$). Hence by Theorem \ref{thm:Lipschitz} and Theorem \ref{thm:NullitycapEquiv}\rf{b2}, the closed set $\partial\Omega$ is $(s,p)$-null if and only if the hyperplane $\R^n_0$ is $(s,p)$-null, which by Theorem \ref{thm:dSet} holds if and only if $s\geq -1/p'$.
\end{proof}

Capacities can rarely be computed exactly. (An exception is provided by Appendix \ref{app:CapacityEquivalence}.) 
But estimates are available for the capacity of balls, which will be of use to us in \S\ref{subsec:Examplessgreaterthanzero}.

\begin{prop}[{\cite[Propositions 5.1.2--4]{AdHe}}]
\label{prop:CapBall}
Let $1<p<\infty$. Given $0<s<n/p$, there exist constants $0<A_{s,p,n}<B_{s,p,n}$, depending on $s$, $p$ and $n$, such that
\begin{align*}
A_{s,p,n} r^{n-sp} \leq \ccap_{s,p}\big(B_r(\bx)\big)\leq B_{s,p,n} r^{n-sp}, \qquad 0<r\leq 1,\,\bx\in\R^n.
\end{align*}
For $s=n/p$, given $c>1$ there exists a constant $C_{c,p,n}>1$, depending on $c$, $p$ and $n$, such that 
\begin{align*}
\frac{1}{C_{c,p,n}}\big(\log(c/r)\big)^{1-p} 
\leq \ccap_{n/p,p}\big(B_r(\bx)\big) 
\leq C_{c,p,n} \big(\log{(c/r)}\big)^{1-p}, \qquad 0<r\leq 1, \,\bx\in\R^n.
\end{align*}
\end{prop}

\subsection{Sets of uniqueness and \texorpdfstring{$(s,p)$}{(s,p)}-nullity} 
\label{subsec:SOU}

We end this section by exploring the relationship between the concept of $(s,p)$-nullity and the concept of sets of uniqueness considered in \cite{AdHe,Maz'ya}. Following \cite[Definition 11.3.1]{AdHe} and \cite[p.~692]{Maz'ya}, given $1<p<\infty$ and $s>0$ we say that $E\subset \R^n$ is a $(s,p)$-set of uniqueness (abbreviated to $(s,p)$-SOU) if 
\[ \{u\in H^{s,p}: \ocap_{s,p}(\supp{u}\cap E^c) =0 \} = \{0\}. \]
Note that if $E$ is Borel then by Theorem \ref{thm:NullitycapEquiv}\rf{b2} this definition can be restated as
\[ \{u\in H^{s,p}: (\supp{u}\cap E^c) \textrm{ is } (-s,p')\textrm{-null} \} = \{0\}. \]
\begin{prop}
\label{prop:SOU}
Let $1<p<\infty$ and $s>0$. 
\begin{enumerate}[(i)]
\item \label{oo} If $E$ is a $(s,p)$-SOU then $E$ is $(s,p)$-null.
\item \label{pp} If $E$ is closed, then $E$ is a $(s,p)$-SOU if and only if $E$ is $(s,p)$-null.
\end{enumerate}
\end{prop}
\begin{proof}
\rf{oo} Suppose that $u\in H^{s,p}$ with $\supp{u}\subset E$. Then $ \ocap_{s,p}(\supp{u}\cap E^c)= \ocap_{s,p}(\emptyset)=0$, and since $E$ is a $(s,p)$-SOU it follows that $u=0$. Hence $E$ is $(s,p)$-null.

\rf{pp} Suppose that $u\in H^{s,p}$ with $\ocap_{s,p}(\supp{u}\cap E^c) =0$. Then $m(\supp{u}\cap E^c) =0$ (this holds e.g.\ by Theorem \ref{thm:CapDimH}\rf{a1}, which gives $\dimH(\supp{u}\cap E^c)<n$), hence $\supp{u}\cap E^c$ is $(s,p)$-null. Since $E$ is closed and $(s,p)$-null, $\supp{u}\cap E$ is also closed and $(s,p)$-null. 
Then Proposition \ref{thm:NullityOfUnions}\rf{ll} gives that $\supp{u}$ is $(s,p)$-null, which implies that $u=0$. Hence $E$ is a $(s,p)$-SOU.
\end{proof}

\section{Examples and counterexamples} 
\label{sec:Domains}
In this section we present examples and counterexamples to illustrate the results of \S\ref{sec:MainResults}. 
\subsection{Boundary regularity and Hausdorff dimension} \label{sec:domainsHausdorff}

The following lemma concerns the relationship between the analytical regularity of the boundary of a set and its Hausdorff dimension. 
Its proof shows how to construct examples of $C^0$ open sets whose boundaries have a given Hausdorff dimension, using the modified Weierstrass-type functions analysed in \cite[Theorem 16.2]{Triebel97FracSpec}. These results should be considered in the context of Theorem \ref{thm:Domains} above. 

\begin{lem}
\label{lem:domains-dimH}
Let $\Omega\subset\R^n$ be an open set such that $\Omega^c$ has non-empty interior. Then:
\begin{enumerate}[(i)]
\item\label{app.a} $n-1\le\dimH(\partial\Omega)\le n$.
\item\label{app.aa} If $\Omega$ is $C^0$ then $m(\partial\Omega)=0$.
\item\label{app.b}  If $\Omega$ is $C^{0,\alpha}$ with $0<\alpha<1$, then $n-1\le\dimH(\partial\Omega)\le n-\alpha$.
\item\label{app.c}  If $\Omega$ is Lipschitz, then $\dimH(\partial\Omega)=n-1$.
\item\label{app.d}  For $n\ge 2$ and $0<\alpha<1$, there exists $\Omega_{\alpha,n}\subset\R^n$ open, bounded and $C^{0,\alpha}$ such that $\dimH(\partial\Omega_{\alpha,n})=n-\alpha$.
\item\label{app.e}  For $n\ge2$, there exists $\Omega_{0,n}\subset\R^n$ open, bounded and $C^0$ such that $\dimH(\partial\Omega_{0,n})=n$.
\end{enumerate}
\end{lem}
\begin{proof}
If $n=1$, there is no distinction between $C^0$, H\"older and Lipschitz open sets: they all mean a countable union of open intervals with pairwise disjoint closures. Hence $\partial \Omega$ contains at most countably many points, so it always has dimension $0=n-1$.
So we assume henceforth that $n\ge2$.

\rf{app.a} The upper bound  $\dimH(\partial\Omega)\le n$ is trivial. For the lower bound, since $\Omega$ is open and its complement $\Omega^c$ has non-empty interior we can take two disjoint balls of some radius $\epsilon>0$ such that $B_\epsilon(\bx)\subset\Omega$ and $B_\epsilon(\by)\subset\Omega^c$.
After translation and rotation, without loss of generality we can assume $\bx=\boldsymbol0$ and $\by=(y,0,\ldots,0)$.
For all $\widetilde\bz\in\R^{n-1}$ with $|\widetilde\bz|<\epsilon$, the point $(0,\widetilde\bz)$ lies in $\Omega$ and the point $(y,\widetilde\bz)$ lies in the interior of $\Omega^c$, so the segment $[(0,\widetilde\bz),(y,\widetilde\bz)]$ contains at least one point in $\deO$.
The orthogonal projection $P_1:\R^n\to\R^{n-1}$ defined by $\bx\mapsto (x_2,\ldots,x_n)$ is Lipschitz and $P_1(\deO)$ contains the set $\{\widetilde\bz\in\R^{n-1}: |\widetilde\bz|<\epsilon\}$ which has Hausdorff dimension equal to $n-1$.
Thus by \cite[Corollary~2.4]{Fal} 
we have $\dimH(\deO)\ge \dimH (P_1(\deO)) \ge n-1$, as required.
\rf{app.aa} The graph of a continuous function has zero Lebesgue measure (this follows from the translation invariance of Lebesgue measure---just consider the measure of the union of infinitely many disjoint vertical translates of the graph). Since the boundary of a $C^0$ open set can be covered by a countable number of $C^0$ graphs, the assertion follows from the countable subadditivity of Lebesgue measure. 
\rf{app.b} If $\partial\Omega$ is the graph of a $C^{0,\alpha}$ function, this follows from \cite[Theorem 16.2(i)]{Triebel97FracSpec}, which states that, for $0<\alpha<1$, the Hausdorff dimension of the graph of a function in $C^{0,\alpha}([-1,1]^{n-1})$ is at most $n-\alpha$.
The general case follows from the ``countable stability'' of the Hausdorff dimension (see \cite[p.~49]{Fal}): for countably many subsets $\{A_j\}_{j\in\N}$ of $\R^n$ we have $\dimH \bigcup_j A_j=\sup_j\dimH F_j$ (note that the open cover of $\deO$ given by the definition of a $C^{0,\alpha}$ open set in \cite[1.2.1.1]{Gri} allows a countable subcover, this can be seen using the compactness of $\deO\cap \overline{B_R(\mathbf{0})}$ for $R>0$ and considering a countable number of balls, e.g.\ $\{\overline{B_\ell(\mathbf{0})}\}_{\ell\in\N}$).
\rf{app.c} Since a Lipschitz open set is $C^{0,\alpha}$ for every $0<\alpha<1$, the Hausdorff dimension of its boundary satisfies $n-1\le\dimH(\partial\Omega)\le n-\alpha$ for all these values, from which the assertion follows.
\rf{app.d} This follows from \cite[Theorem 16.2(ii)]{Triebel97FracSpec} which provides a function $f_{\alpha,n}:[-1,1]^{n-1}\to[-1,1]$ of class $C^{0,\alpha}$ (defined via a modification of the well-known Weierstrass function),
 whose graph has Hausdorff dimension exactly equal to $n-\alpha$ (that $f_{\alpha,n}$ can be taken with values in $[-1,1]$ follows because affine transformations do not affect the Hausdorff dimension).
We can immediately define $\Omega_{\alpha,n}:=\{\bx\in(-1,1)^{n-1}\times\R,\; -2 <x_n<f_{\alpha,n}(x_1,\ldots,x_{n-1})\}$.
\rf{app.e} We construct $\Omega_{0,n}$ by gluing together the functions from the previous step.
We define $f_{0,n}:[-1,1]^{n-1}\to[-1,1]$ as
$f_{0,n}(x_1,\ldots,x_{n-1}):=f_{1/j,n}(x_1,\ldots,2^jx_{n-1}-1)$ if $2^{-j}<x_{n-1}\le 2^{-j+1}$ and $j\in\N$. To ensure global continuity we assume that each $f_{1/j,n}$ vanishes at $x_{n-1}=\pm 1$ (if necessary we can ensure this by multiplying $f_{1/j,n}$ by a suitable element of $\scrD((-1,1))$). We then construct the corresponding open set as before.
Its boundary has Hausdorff dimension not smaller than $n-1/j$ for all $j\in\N$, from which the assertion follows.
\end{proof}

\subsection{Swiss-cheese and Cantor sets}
\label{subsec:SwissCheese}

Our remaining examples belong to two classes of compact sets with empty interior: ``Swiss-cheese'' sets and Cantor sets. 
When $n=1$, Cantor sets are special cases of Swiss-cheese sets.

\begin{defn}\label{def:Gruyere}
A Swiss-cheese set is a non-empty compact set $K$ with empty interior, constructed as
$K=\overline{\Omega}\setminus (\bigcup_{i=1}^\infty B_{r_i}(\bx_i))$,
where $\Omega\subset\R^n$ is a bounded open set, 
and $\bx_i\in \R^n$ and $r_i>0$ for each $i\in \N$.
\end{defn}
Given $1<p<\infty$ and $s\in[-n/p',n/p]$, one would like to engineer the $(s,p)$-nullity (or otherwise) of $K$ by choosing an appropriate sequence of ball centres $\{\bx_i\}_{i=1}^\infty$ and a sequence of radii $\{r_i\}_{i=1}^\infty$ which tends to zero at an appropriate rate as $i\to\infty$. 

In one dimension ($n=1$), a well-studied family of examples of this construction is the classical ternary Cantor set and its generalisations.
These are Swiss-cheese sets for which for every $i\in\N$ the centre $\bx_i$ of the $i$th subtracted ball is chosen as the centre of one of the largest connected components of $\overline\Omega\setminus (\bigcup_{i'=1}^{i-1} B_{r_{i'}}(\bx_{i'}))$. 
Following \cite[\S5.3]{AdHe}, we consider Cantor sets in $\R^n$ (sometimes known as ``Cantor dust'' when $n\geq 2$) defined in the following general way. 

\begin{defn}\label{def:Cantor}
Let $\{ l_j\}_{j=0}^\infty$ be a decreasing sequence of positive numbers such that $0<l_{j+1}<l_j/2$ for $j\geq 0$. 
Let $E_0=[0,l_0]$ and for $j\geq0$ let $E_{j+1}$ be constructed from $E_j$ by removing an open interval of length $l_j-2l_{j+1}$ from the middle of each of the $2^j$ subintervals making up $E_j$, each of which has length $l_j$. 
The Cantor set associated with the sequence $\{ l_j\}_{j=1}^\infty$ is then 
defined to be $E:=\bigcap_{j=0}^\infty E_j$.
For $n\geq 1$ we define the $n$-dimensional Cantor set $E^{(n)}\subset\R^n$ to be the Cartesian product of $n$ copies of~$E$. 
\end{defn}
The total measure of $E_j$ is $2^jl_j$, so the measure of $E$ is $\lim_{j\to\infty}2^jl_j\in[0,l_0)$.
Since each $E_j$ is compact, $E$ is also compact. 
Moreover, $E$ has empty interior and is uncountable. 
Without loss of generality we shall always take $l_0=1$, so that $E\subset[0,1]$ and $E^{(n)}\subset [0,1]^n$.

\subsubsection{The case \texorpdfstring{$s<0$}{s negative} (sets with zero measure)}
\label{subsec:Examplesslessthanzero}

For the case $-n/p'\leq s<0$ the $(s,p)$-nullity of a Cantor set $E^{(n)}$ can be characterised precisely in terms of the asymptotic behaviour of the sequence $\{l_j\}_{j=1}^\infty$.
Theorem~5.3.2 in \cite{AdHe}, which was first proved in \cite[Theorem~7.4]{MazyaHavin72} (see also \cite[\S10.4.3, Proposition~5]{Maz'ya}),
provides necessary and sufficient conditions for $\ccap_{s,p}(E^{(n)})$ to vanish, for a given $0<s\leq n/p$. 
Reinterpreting this result in terms of $(s,p)$-nullity, using Theorem \ref{thm:NullitycapEquiv}\rf{b2}, provides the following result.

\begin{thm}\label{thm:CantorCapacity}
Let $1<p<\infty$ and $-n/p'\leq s<0$. The Cantor set $E^{(n)}$ is $(s,p)$-null if and only if
\begin{align*}
\sum_{j=0}^\infty \left(2^{-jn}l_j^{-(sp'+n)}\right)^{p-1} &= \infty, \qquad \textrm{for }-n/p'<s<0,
\\
\text{and if and only if}\qquad
\sum_{j=0}^\infty 2^{-jn(p-1)}\log{1/l_j} &= \infty, \qquad \textrm{for }s=-n/p'.
\end{align*}
\end{thm}

Using Theorem~\ref{thm:CantorCapacity} we can construct a zoo of Cantor sets, all with zero measure, which realise all of the possible nullity sets permitted by Corollary \ref{cor:NullitySets}.

\begin{thm}\label{thm:CantorZoo}
Let $n\in\N$, $0\le\dd<n$ and $1<\pp<\infty$.
Then there exists two Cantor sets $E^{(n)}_{\dd ,\pp},F^{(n)}_{\dd ,\pp}\subset\R^n$ such that $\dimH E^{(n)}_{\dd ,\pp}= \dim_H F^{(n)}_{\dd ,\pp}=\dd $ and 
\begin{align*}
&E^{(n)}_{\dd ,\pp} \text{ is $(s,p)$-null if and only if either }\;
s>(\dd -n)/p', \quad \text{or}\quad s=(\dd -n)/p' \text{ and } p\le\pp,\\
&F^{(n)}_{\dd ,\pp} \text{ is $(s,p)$-null if and only if either }\;
s>(\dd -n)/p', \quad \text{or}\quad s=(\dd -n)/p' \text{ and } p< \pp.
\end{align*}
Furthermore, there exist $F^{(n)}_{\dd ,\infty},F^{(n)}_{\dd ,1},F^{(n)}_{n ,\infty}\subset\R^n$ such that $\dimH F^{(n)}_{\dd,\infty}= \dim_H F^{(n)}_{\dd,1}=\dd$, $\dim_H F^{(n)}_{n,\infty}=n$, and 
\begin{align*}
&F^{(n)}_{\dd,\infty} \text{ is $(s,p)$-null if and only if }\; s\ge(\dd -n)/p',&&0\le\dd\le n,\\
&F^{(n)}_{\dd,1} \text{ is $(s,p)$-null if and only if }\; s>(\dd -n)/p',&&0\le\dd<n.
\end{align*}
Moreover, for any Borel set $E\subset\R^n$ with $m(E)=0$, 
(exactly) one of the $n$-dimensional Cantor sets above is $(s,p)$-null precisely for the same values $s,p$ for which $E$ is $(s,p)$-null.
\end{thm}
\begin{proof}
To construct each Cantor set, setting $l_0=1$ as usual, we just need to find a sequence $\{l_j\}_{j=1}^\infty$ with $0<l_1<1/2$ and $0<l_{j+1}<l_j/2$ for $j\ge1$, such that the series in Theorem~\ref{thm:CantorCapacity} diverge for the desired set of values of $s$ and $p$.
The 9 different examples we require are defined in Table \ref{tab:CantorProof}.

\begin{table}[ht!]
\caption{\label{tab:CantorProof} Cantor sets used in the proof of Theorem \ref{thm:CantorZoo}.}
\begin{center}
\begin{tabular}{|l|l|l|l|}\hline
For & and & choosing $l_0=1$ and, for $j\geq1$, & gives\\
\hline
\hline
$\dd=0$
& $1<\pp<\infty$ 
& $l_j=2^{-2(2^{jn(\pp-1)}-1)/(2^{n(\pp-1)}-1)}$
&  $E^{(n)}_{0,\pp}$
\\[1.5mm]
\hline
$ 0<\dd<n$
& $1<\pp<\infty$ 
& $l_j=2^{-jn/\dd} \big(1+\frac{(j-1)}{2}(2^{(n-\dd)(\pp-1)}-1)\big)^{\frac1{\dd(\pp-1)}}$,
&  $E^{(n)}_{\dd,\pp}$ 
\\
\hline\hline
$\dd=0$
& $\pp=1$ 
& $2^{-j^2-1}$
&  $F^{(n)}_{0,1}$
\\[1.5mm]
\hline
$\dd=0$
& $1<\pp<\infty$ 
&  $l_j=2^{-2^{(j+j_0)n(\pp-1)}/(j+j_0)^2}$
&  $F^{(n)}_{0,\pp}$
\\
&& $j_0=\min\{j\in \N$ s.t.\ $\frac{2^{(j+1)n(\pp-1)}}{(j+1)^2}-\frac{2^{jn(\pp-1)}}{j^2}
>1\}$
&\\
\hline
$\dd=0$
& $\pp=\infty$ 
& $l_j=2^{-2^{2^j}}$
&  $F^{(n)}_{0,\infty}$
\\[1.5mm]
\hline
$ 0<\dd<n$
& $\pp=1$ 
& $l_j=2^{-jn/\dd}\,2^{(n/\dd-1)\sqrt j/2}$ 
& $F^{(n)}_{\dd,1}$ 
\\[1.5mm]
\hline
$ 0<\dd<n$
& $1<\pp<\infty$ 
& $l_j
=2^{-jn/\dd}\left((j+j_0)\log^2(j+j_0)\right)^{\frac{1}{\dd(\pp-1)}}$
&  $F^{(n)}_{\dd,\pp}$ \\
&&
$j_0=\min\big\lbrace 2\le j\in \N$ s.t.\ 
$\frac{(j+1)\log^2(j+1)}{j\log^2j}<2^{(n-\dd)(\pp-1)} $
&\\
&&\qquad\qquad and 
$2^{-jn(\pp-1)}(j+1)\log^2(j+1)<2^{(n-\dd)(\pp-1)} \big\rbrace$
&\\
\hline
$ 0<\dd<n$
& $\pp=\infty$ 
& $l_j=2^{-jn/\dd}$
&  $F^{(n)}_{\dd,\infty}$
\\[1.5mm]
\hline
$ \dd=n$
& $\pp=\infty$ 
& $l_j=2^{-j}/(j+1)$
&  $F^{(n)}_{n,\infty}$
\\[1.5mm]
\hline
\end{tabular}
\end{center}
\end{table}

The convergence of the corresponding series can be verified using the fact that, for $x,y>0$,
the geometric series $\sum_{j\ge 0} x^j<\infty\Leftrightarrow x<1$;
its generalisation $\sum_{j\ge 0} x^jy^{\sqrt j}<\infty\Leftrightarrow x<1$ or $x=1$ and $y<0$;
the generalised harmonic series $\sum_{j\ge 1} j^{-x}<\infty\Leftrightarrow x>1$; 
and 
$\sum_{j\ge 2} j^{-x}\log^{-y} j<\infty\Leftrightarrow x>1$ or $x=1$ and $y>1$.
To prove that $F^{(n)}_{n,\infty}$ is $(0,p)$-null, we compute the measure $m(F^{(n)}_{n,\infty})=\lim_{j\to\infty}2^{nj}l_j^n=0$. 
The final statement in the assertion follows from Corollary \ref{cor:NullitySets}.
\end{proof}

A few remarks are in order here. First, the standard one-third Cantor set with $l_j=1/3^j$ corresponds to $F^{(1)}_{\log2/\log3, \infty}$ in the table in the proof of Theorem~\ref{thm:CantorZoo}.
Second, the set $F^{(n)}_{0,\infty}$ demonstrates that Corollary \ref{cor:Countable} is not sharp: there exist uncountable sets which are $(-n/p',p)$-null for all $1<p<\infty$. 
Third, Theorem~\ref{thm:dSet} implies that none of the Cantor sets constructed in Theorem~\ref{thm:CantorZoo} are $d$-sets, except for $F^{(n)}_{\dd,\infty}$, $0\le \dd\le n$.

\subsubsection{The case \texorpdfstring{$s>0$}{s positive} (sets with non-zero measure)}
\label{subsec:Examplessgreaterthanzero}

For a general Swiss-cheese set defined as in Definition~\ref{def:Gruyere} we can give sufficient conditions for non-$(s,p)$-nullity, for $0<s\leq n/p$, by combining Theorem \ref{thm:SOU} with capacity estimates from \S\ref{sec:Capacity}.
\begin{thm}
\label{thm:SwissCheeseNullity}
Let $K=\overline{\Omega}\setminus (\bigcup_{i=1}^\infty B_{r_i}(\bx_i))$ be a Swiss-cheese set defined as in Definition~\ref{def:Gruyere}. Suppose that $0<r_i\leq 1$ for each $i\in \N$, and let $0<r\leq 1$ be such that $B_r(\bx)\subset \Omega$ for some $\bx\in \Omega$. Given $1<p<\infty$ and $0<s\leq n/p$, let $0<A_{s,p,n}\leq B_{s,p,n}$, $c>1$ and $C_{c,p,n}>1$ denote the constants from Proposition \ref{prop:CapBall}. Then $K$ is not $(s,p)$-null provided that
\begin{align}
\label{eqn:sumcond1}
\sum_{i=1}^\infty r_i^{n-sp} &< \frac{A_{s,p,n}}{B_{s,p,n}}r^{n-sp},  &&\textrm{for } 0<s<n/p,
\\
\label{eqn:sumcond2}
\text{and}\quad
\sum_{i=1}^\infty\big(\log (c/r_i)\big)^{1-p} &< \frac{\big(\log (c/r)\big)^{1-p} }{C_{c,p,n}^2},&& \textrm{for } s=n/p.
\end{align}
\end{thm}
\begin{proof}
By countable subadditivity (Proposition \ref{thm:CountableSubadd}), we have that 
\begin{align*}
\ccap_{s,p}(B_r(\bx)\setminus K) \leq \ccap_{s,p}\left(\bigcup_{i=1}^\infty B_{r_i}(\bx_i)\right) \leq \sum_{i=1}^\infty \ccap_{s,p}\big(B_{r_i}(\bx_i)\big).
\end{align*}
The statement of the theorem then follows from Theorem \ref{thm:SOU} (in particular the equivalence \rf{a3}$\Leftrightarrow$\rf{c3}) and the estimates in Proposition \ref{prop:CapBall}.
\end{proof}

As was alluded to in \S\ref{subsec:ResultsSG0}, given $n\in\N$, $1<p<\infty$ and $0<s\leq n/p$, Theorem \ref{thm:SwissCheeseNullity} allows us to construct non-empty compact Swiss-cheese sets $K\subset \R^n$ with empty interior, which are not $(s,p)$-null. To ensure that $K$ has empty interior, one can take the ball centres $\{\bx_i\}_{i=1}^\infty$ to be any countable dense subset of $\Omega$. To ensure that $K$ is non-empty, one just needs to make sure that $r_i\to 0$ sufficiently fast as $i\to\infty$ so that $\sum_{i=1}^\infty m\big(B_{r_i}(\bx_i)\big) < m(\overline{\Omega})$, from which it follows that $m(K)>0$.

As further concrete examples we consider two families of Cantor sets in $\R$, for which we prove non-$(s,p)$-nullity for certain values of $s>0$ in Proposition~\ref{prop:FatCantorCapacity}. 
Note that since Theorem~\ref{thm:SwissCheeseNullity} concerns Swiss-cheese sets, it applies to Cantor sets only for $n=1$.

\begin{example}[``Fat'' Cantor set]
\label{ex:FatCantor}
Given $0<\alpha<1/2$ and $0<\beta<1-2\alpha$, denote by $G_{\alpha,\beta}^{(n)}\subset \R^n$, $n\in\N$, the Cantor set constructed as in Definition~\ref{def:Cantor} with $l_{j+1}=1/2(l_j-\beta \alpha^j)$ for $j\geq 0$.
For $n=1$, the resulting set $G_{\alpha,\beta}$ is called a ``generalised Smith--Volterra--Cantor'' set in \cite[\S2.4-2.5]{DiMartinoUrbina}, 
where it is denoted $SVC(\alpha,\beta)$. The classical Smith--Volterra--Cantor set 
is obtained by the choice $\alpha=\beta=1/4$. (We remark also that the ``$\epsilon$-Cantor'' sets of \cite[p.~140]{AliBur} are obtained by the choice $\beta=(1-\epsilon)/2$, $\alpha=1/4$, for a given $0<\epsilon<1$.)
Then
\begin{align*}
l_j=\frac{1}{2^j}\left(1-\beta \left(\frac{1-(2\alpha)^j}{1-2\alpha} \right)\right),
\quad\text{so that}\quad
2^jl_j = 1-\beta \left(\frac{1-(2\alpha)^j}{1-2\alpha} \right) \to 1-\frac{\beta}{1-2\alpha}, \qquad j\to\infty,
\end{align*}
from which we conclude that $m(G_{\alpha,\beta}^{(n)})=(1-\beta/(1-2\alpha))^n>0$.
(In particular for the classical Smith--Volterra--Cantor set 
we have $m(G_{1/4,1/4})=1/2$.)
Hence $G_{\alpha,\beta}^{(n)}$ is not $(s,p)$-null for any $s\leq 0$ and any $1<p<\infty$;
on the other hand, since it has empty interior, $G_{\alpha,\beta}^{(n)}$ is $(s,p)$-null for $s>n/p$.
\end{example}

\begin{example}[``Super-fat'' Cantor set]
\label{ex:SuperFatCantor}
To construct a Cantor set which is even fatter (in the sense of nullity), choose
$l_j=2^{-j}(1-\gamma+\gamma^{\delta^j})$
with $0<\gamma<1$ and $\delta>1$. 
Then $l_{j+1}<l_j/2$ for all $j\ge0$, and when constructing level $j+1$ from level $j\in\N_0$ we subtract $2^j$ intervals with radii $\frac12 l_{j}-l_{j+1}=2^{-j-1}(\gamma^{\delta^j}-\gamma^{\delta^{j+1}})<2^{-j-1}\gamma^{\delta^j}$, which decrease faster as $j\to\infty$ than those in Example~\ref{ex:FatCantor}.
We denote the resulting Cantor set, which has Lebesgue measure $0<(1-\gamma)^n<1$, by $K_{\gamma,\delta}^{(n)}$.
\end{example}

\begin{prop}
\label{prop:FatCantorCapacity}
Given $1<p<\infty$ and $0<s<1/p$, there exist $0<\alpha<1/2$ and $0<\beta<1-2\alpha$ such that the fat Cantor set $G_{\alpha,\beta}\subset\R$ defined in Example~\ref{ex:FatCantor} is not $(s,p)$-null.

Given $1<p<\infty$, there are parameters $0<\gamma<1$ and $\delta>2^{\frac1{p-1}}$ such that the corresponding super-fat Cantor set $K_{\gamma,\delta}\subset\R$ of Example~\ref{ex:SuperFatCantor} is not $(1/p,p)$-null.
\end{prop}
\begin{proof}
To analyse the fat Cantor set using Theorem \ref{thm:SwissCheeseNullity} we note that in constructing $E_j$ from $E_{j-1}$ we remove $2^{j-1}$ intervals of radius $(\beta/2)\alpha^{j-1}$. The intervals we remove are all disjoint, so 
\begin{align*}
\sum_{i=1}^\infty r_i^{n-sp} = \sum_{j=0}^\infty 2^j \left(\frac{\beta \alpha^j}{2}\right)^{1-sp}=\left(\frac{\beta}{2}\right)^{1-sp}\sum_{j=0}^\infty(2\alpha^{1-sp})^j 
= \frac{\left(\beta/2\right)^{1-sp}}{1-2\alpha^{1-sp}},
\end{align*}
where for the sum to converge we need $2\alpha^{1-sp}<1$, i.e.\ 
\begin{align*}
0<\alpha<2^{-1/(1-sp)}, \quad \textrm{or, equivalently,} \quad 
s<s_{\alpha,p}:=\frac{1}{p}\left( 1+\frac{\log{2}}{\log{\alpha}} \right).
\end{align*}
Since $l_0=1$ we can choose $r=1/2$ in Theorem~\ref{thm:SwissCheeseNullity},
then \rf{eqn:sumcond1} will be satisfied (and hence $G_{\alpha,\beta}$ is not $(s,p)$-null) provided that $\beta$ lies in the range
\begin{align*}
 0<\beta< \left(\frac{A_{s,p,1}}{B_{s,p,1}}\right)^{1/(1-sp)}(1-2\alpha^{1-sp})^{1/(1-sp)}.
\end{align*} 

To analyse the super-fat Cantor set one proceeds in a similar way using \eqref{eqn:sumcond2} instead of \eqref{eqn:sumcond1}.
Possible parameters ensuring that $K_{\gamma,\delta}$ is not $(1/p,p)$-null are
$$ 
\delta> 2^{\frac1{p-1}} \text{ and }
0<\gamma<(2c)^{{\big(-C_{c,p,1}^{-2}(1-2\delta^{1-p})\big)}^{\frac1{1-p}}}<1.
$$
\end{proof}
By mimicking the approach of Theorem~\ref{thm:CantorZoo} (in particular the construction of the set $F^{(n)}_{0,\infty}$), one might try to construct an ``ultra-fat'' Cantor set that is not $(s,p)$-null for all $s\le 1/p$ and $1<p<\infty$.
We conjecture that such a set can be constructed by subtracting at every level $j\in\N_0$ an interval with radius proportional to $1/2^{2^{2^j}}$.
However, to use Theorem \ref{thm:SwissCheeseNullity} to prove non-nullity, one needs to know the dependence on $p$ of the constant $C_{c,p,1}$ from Proposition~\ref{prop:CapBall}, which corresponds to the constant $A$ in \cite[Proposition~5.1.3]{AdHe}.

The proof of Proposition \ref{prop:FatCantorCapacity} implies that if $0<\alpha<1/2$ is fixed, then the fat Cantor set defined in Example \ref{ex:FatCantor} is not $(s,p)$-null provided that 
$s< s_{\alpha,p}$ and that $\beta$ is  sufficiently small. 
But the permitted range of $\beta$ appears to depend on $s$ and $p$ in a nontrivial way. The following proposition provides a complementary result for the special case of $(s,2)$-nullity, which removes the restriction on $\beta$. 
Our proof, which is independent of the capacity-based approach of Theorem~\ref{thm:SwissCheeseNullity}, is inspired by the method used to prove similar results in \cite[Lemma 2.5 and Theorem 3.1]{La:11}, and involves estimating the local $L^2$ norm of the Fourier transform of the characteristic function of $E$ 
in terms of the $L^2$ norm of the difference between the characteristic function and a slightly shifted version of it.

\begin{prop}
\label{prop:FatCantor}
Let $E\subset\R$ be a Cantor set as in Definition \ref{def:Cantor}.
Let $\gap_j:=l_{j-1}-2l_j$, $j\geq 1$, denote the length of the $2^{j-1}$ gaps introduced in $E_{j-1}$ to construct $E_j$.
Assume that $\gap_j<l_j$ (equivalently $3l_j>l_{j-1}$) and that $\gap_j<\gap_{j-1}$ for all $j\ge1$. 
If 
\begin{align}\label{eq:GapSum}
\sum_{j\ge2} \frac{\gap_{j-1}^{2-2s}}{\gap_j^{2}}\Big(\sum_{k\ge j}2^{k}\gap_k\Big)<\infty
\end{align}
for $s>0$, then the characteristic function $\chi_E$ of $E$ belongs to $H^{s,2}$ (and hence $E$ is not $(s,2)$-null).

In particular, for the fat Cantor set $G_{\alpha,\beta}$ with $0<\alpha<1/2$ and $0<\beta<1-2\alpha$, $\chi_{G_{\alpha,\beta}}\in H^{s,2}$ for all $s<{s}_{\alpha,2}$, 
where
\[{s}_{\alpha,2}:=\frac12\Big(1+\frac{\log 2}{\log \alpha}\Big)\in(0,1/2).\]
\end{prop}

\begin{proof}
To prove that $\chi_E\in H^{s,2}$ for the claimed range of $s$ our strategy is to prove directly that $\|\chi_E\|_{H^{s,2}}^2=\int_{-\infty}^\infty(1+|\xi|^2)^s|\hat\chi_E(\xi)|^2 \rd \xi<\infty$ (cf.\ \rf{eqn:Hs2NormFourier}). Since the characteristic function $\chi_E$ is compactly supported, $\hat\chi_E$ is analytic, and so it suffices to control the behaviour of the integrand at infinity.

Since $\{\gap_j\}_{j=1}^\infty$ is a non-increasing sequence, the gaps between the subintervals in $E_j$ have minimal length $\min_{1\le k\le j}\gap_k=\gap_j$. 
For any such Cantor set $E$, 
the symmetric difference between $E_j$ and its translation $\{E_j+\gap_j\}$ is composed of $2^{j+1}$ (some open, some semi-open) intervals of length $\gap_j$ (in some cases touching each other at one extreme). 
So $\|\chi_{E_j}-\chi_{\{E_j+\gap_j\}}\|^2_{L^2(\R)} = 2^{j+1}\gap_j$.
Comparing $E$ and $E_j$ we have 
$\|\chi_{E_j}-\chi_E\|^2_{L^2(\R)} = |E_j|-|E|=\sum_{k>j}2^{k-1}\gap_k$ and similarly for the translates.
Using the triangle inequality we have
$$\|\chi_E-\chi_{\{E+\gap_j\}}\|_{L^2(\R)}^2
\le \sum_{k\ge j}2^{k+2}\gap_k.$$
The intervals
$I_j:=(\gap_{j-1}^{-1},\gap_j^{-1}]$, $j\geq 2$, 
are a partition of a neighbourhood of infinity; explicitly, $\bigcup_{j\ge2} I_j=(\gap_1,\infty)$ with empty mutual intersections. 
We define the coefficients
\begin{align*}
z_j:=\sup_{\xi\in I_j} \frac{\xi^{2s}}{1-\cos (\gap_j\xi)}
\le\sup_{\xi\in I_j}c_1 \frac{\xi^{2s}}{(\gap_j\xi)^2}
\le c_1\frac{\gap_{j-1}^{2-2s}}{\gap_j^2}, \qquad j\ge2,
\end{align*}
where $c_1:=1/(1-\cos1)\approx2.175$.
From Plancherel's Theorem and the above bounds we have
\begin{align*}
\int_{|\xi|>\gap_1}(1+|\xi|^2)^s|\hat\chi_E(\xi)|^2\rd\xi
&\le 2^s\sum_{j\ge2}\int_{|\xi|\in I_j}|\xi|^{2s}|\hat\chi_E(\xi)|^2\rd\xi\\
&\le 2^s\sum_{j\ge2} z_j 
\int_{|\xi|\in I_j} \big(1-\cos(\gap_j\xi)\big)|\hat\chi_E(\xi)|^2\rd\xi\\
&\le
2^{s-1}\sum_{j\ge2} z_j 
\int_\R |1-\re^{-\ri\xi\gap_j}|^2 \, |\hat\chi_E(\xi)|^2\rd\xi\\
&= 2^{s-1}\sum_{j\ge2} z_j \|\chi_E-\chi_{\{E+\gap_j\}}\|_{L^2(\R)}^2\\
&\le  2^{s+1}\sum_{j\ge2} z_j
\Big(\sum_{k\ge j}2^{k}\gap_k\Big)
\\
&\le  2^{s+1}c_1 \sum_{j\ge2} \frac{\gap_{j-1}^{2-2s}}{\gap_j^{2}}
\Big(\sum_{k\ge j}2^{k}\gap_k\Big).
\end{align*}
Thus, if condition \eqref{eq:GapSum} holds, $\|\chi_E\|_{H^{s,2}}^2=\int_{-\infty}^\infty(1+|\xi|^2)^s|\hat\chi_E(\xi)|^2 \rd \xi<\infty$ and the proof is complete.

In the fat Cantor case we have $\gap_j=\beta\alpha^{j-1}$, so $\{\gap_j\}_{j=1}^\infty$ is strictly decreasing and $\gap_j<l_j$ for each $j\geq 1$. 
Hence
\begin{align*}
\int_{|\xi|>\gap_1}(1+|\xi|^2)^s|\hat\chi_E(\xi)|^2\rd\xi  
&\le  \frac{2^{s+1}c_1 \beta^{1-2s}}{\alpha^{3-4s}}\sum_{j\ge2} \alpha^{-2sj}\Big(\sum_{k\ge j}(2\alpha)^{k}\Big) 
\le \frac{2^{s+1}c_1 \beta^{1-2s}}{\alpha^{3-4s}(1-2\alpha)}\sum_{j\ge2} (2\alpha^{1-2s})^j,
\end{align*}
which is bounded for $s<s_{\alpha,2}$.
\end{proof}

Proposition \ref{prop:FatCantor} demonstrates the non-$(s,2)$-nullity of the set $G_{\alpha,\beta}$, for $s< s_{\alpha,2}$, by showing that a specific function supported inside $G_{\alpha,\beta}$ (namely $\chi_{G_{\alpha,\beta}}$) belongs to $H^{s,2}$. We note that the proof relies on Plancherel's Theorem, and hence does not generalise easily to the case $p\neq 2$. However, we can immediately deduce some simple consequences for the case $p\neq 2$. 
Since $G_{\alpha,\beta}$ is bounded, $\chi_{G_{\alpha,\beta}}$ belongs to $L^p$ for all $1<p<\infty$ and, by interpolation~\eqref{eq:Interpolation}, it belongs to $\Hsp$ for $2\le p<\infty$ and $s< s_{\alpha,p}=2 s_{\alpha,2}/p$.
Hence the nullity threshold of $G_{\alpha,\beta}$ satisfies $s_{G_{\alpha,\beta}}(p)\geq {s}_{\alpha,p}$ for all $2\le p<\infty$, and all $0<\beta<1-2\alpha$. 
In the light of Proposition \ref{prop:FatCantorCapacity} we expect that this result also extends to $1<p<2$, but we do not have a proof of this.  (The only (
weaker) $\beta$-independent result we have for $1<p<2$ is that $s_{G_{\alpha,\beta}}(p)\geq {s}_{\alpha,2}$ for all $1<p<2$, which follows from the embedding \eqref{eq:EmbeddingCompact}.)
One might speculate that these inequalities are actually equalities, i.e.\ that $s_{G_{\alpha,\beta}}(p)= {s}_{\alpha,p}$ for all $n\in\N$ and $1< p<\infty$. But we do not know of any techniques for obtaining an upper bound on $s_{G_{\alpha,\beta}}(p)$ that would verify this.
As far as we are aware, calculating the exact value of $s_{G_{\alpha,\beta}}(p)$ is an \textbf{open problem}.

Regarding the higher-dimensional fat Cantor sets $G_{\alpha,\beta}^{(n)}$, $n\geq 2$, it follows from the above results, combined with Proposition~\ref{cor:TensorProduct}, 
that $s_{G_{\alpha,\beta}^{(n)}}(p)\geq {s}_{\alpha,p}$ for all $2\le p<\infty$, and $s_{G_{\alpha,\beta}^{(n)}}(p)\geq {s}_{\alpha,2}$ for all $1<p<2$. 
The characteristic function of the one-dimensional fat Cantor set $G_{\alpha,\beta}$ does not belong to $\Hsp(\R)$ for any $s>1/p$ by the Sobolev embedding $\Hsp(\R)\subset C^0(\R)$ (or equivalently by the $(s,p)$-nullity of $G_{\alpha,\beta}$). 
Hence Proposition~\ref{prop:TensorNegative} in the Appendix implies that for all $n\ge2$ the characteristic function $\chi_{G_{\alpha,\beta}^{(n)}}$ of the corresponding set $G_{\alpha,\beta}^{(n)}$ does not belong to $\Hsp(\R^n)$ either, for $s>1/p$, since it can be written as the tensor product $\chi_{G_{\alpha,\beta}^{(n)}}=\chi_{G_{\alpha,\beta}^{(1)}}\otimes\chi_{G_{\alpha,\beta}^{(n-1)}}$ due to the Cartesian-product structure of Cantor sets.
Thus we cannot hope to extend the proof of Proposition~\ref{prop:FatCantor} to show that 
$G_{\alpha,\beta}^{(n)}$ is not $(s,p)$-null for some $1/p<s\le n/p$.
Indeed, for any Cantor set $E^{(n)}\subset\R^n$, if $1/p<s\le n/p$ then non-zero functions in $\Hsp_{E^{(n)}}$ (if any exist) cannot be of tensor-product form.

\section{Conclusion}\label{sec:conclusion}
In Definition \ref{def:nullity} and Proposition \ref{prop:NullityThreshold} we introduced the concepts of ``$(s,p)$-nullity'' and the ``nullity threshold'' $s_E(p)$ of a subset $E\subset\R^n$, to describe the Sobolev regularity of the distributions supported by $E$. 
We now summarise our contributions to the questions \textbf{Q1}--\textbf{Q3} posed in \S\ref{subsec:nullity}. 
(An even more concise summary is given in tabular form in Table \ref{tab:Summary}).
We assume throughout this section that $E$ is Borel with empty interior. 
Our first observation is that the case $m(E)=0$ is significantly better understood than the case $m(E)>0$.

\paragraph{\textbf{Q1:} Given $1<p<\infty$ and $E\subset \R^n$ with empty interior, can we determine $s_E(p)$?}

If $m(E)=0$ then $s_E(p)$ is immediately computed in terms of Hausdorff dimension as $s_E(p)=(\dimH E-n)/p'$, see Theorem~\ref{thm:NullityHausdorff}.
If $m(E)>0$ then we know $0\le s_E(p)\le n/p$, but the only general result we know of that allows a more precise characterisation of $s_E(p)$ is Theorem~\ref{thm:SOU}, which appears useful only for proving lower bounds on $s_E(p)$. 
Using a result from \cite{Po:72a} we described two types of set with maximum nullity threshold $s_E(p)=n/p$ for all $1<p<\infty$, see Corollary~\ref{cor:Compact} and Theorem~\ref{thm:OpenMinusDense}.
We also derived lower bounds on $s_E(p)$ for some ``fat'' and ``super-fat'' Cantor sets and Swiss cheese sets in \S\ref{subsec:Examplessgreaterthanzero} (the lower bound obtained for the fat Cantor set is represented by the dotted line in Figure~\ref{fig:Parallelogram}). 
But we are unaware of any viable techniques for obtaining nontrivial upper bounds on $s_E(p)$ when $m(E)>0$. In fact, we have no evidence whatsoever of the existence of sets with $s_E(p)\in(0,n/p)$. 

\textbf{Open question:} Given $1<p<\infty$, do there exist sets $E$ for which $s_E(p)\in(0,n/p)$? 

\textbf{Open question:} What is the nullity threshold of the fat Cantor sets of Example \ref{ex:FatCantor}? 

\paragraph{\textbf{Q2:} For which functions $f:(1,\infty)\to[-n,n]$ does there exist $E\subset\R^n$ such that $f(p)=s_E(p)$ for all $p\in(1,\infty)$?}

To describe the possible nullity-threshold functions, it is convenient to make the change of variable $p\mapsto r=1/p$, and define, for $E\subset\R^n$ non-empty with empty interior,
$$S_E(r):=s_E(1/r)=\inf\big\{s\in\R, \text{ such that }E \text{ is } (s,1/r)\text{-null}\big\},\qquad 0<r<1.$$
The function $S_E$ satisfies the following conditions:
\begin{itemize}
\item $n(r-1)\le S_E(r)\le nr$ for $0<r<1$, so the graph of $S_E(r)$ in the $rs$-plane lies in the shaded parallelogram in Figure~\ref{fig:Parallelogram} (cf.\ Lemma \ref{lem:nullity1}\rf{dd}--\rf{jj}).
\item $S_E$ is non-decreasing and Lipschitz-continuous with $0\le S'_E(r)\le n$ for a.e.\ $0<r<1$ (cf.\ Lemma \ref{lem:nullity1}\rf{bb}).
\item The graph of $S_E$ in the $rs$-plane cannot cross the line $s=0$. Precisely, if $S_E(r^*)\ge0$ (or $S_E(r^*)<0$) for some $r^*\in(0,1)$ then $S_E(r)\ge0$  (or $S_E(r)<0$, respectively) for all $r\in(0,1)$ (cf.\ Lemma \ref{lem:nullity1}\rf{qq}).
\item If $m(E)=0$ then $S_E(r)=(n-\dim_H E)(r-1)$ (cf.\ Theorem~\ref{thm:NullityHausdorff}). 
So if the graph of $S_E$ lies in the lower half $rs$-plane in Figure~\ref{fig:Parallelogram}, it is necessarily a straight line through $(1,0)$ with slope equal to $n-\dimH E$.
\end{itemize}

Hence the only functions $F:(0,1)\to[-n,0]$ which can be realised as $F(r)=S_E(r)$ for some $E\subset\R^n$ are straight lines $F(r)=\lambda (r-1)$, $\lambda\in[0,n]$, 
which are realised by any set $E$ with $m(E)=0$ and $\dimH{E}=n-\lambda$ (for example the Cantor sets of Theorem~\ref{thm:CantorZoo}). By contrast, the only functions $F:(0,1)\to[0,n]$ that we can provably realise as $F(r)=S_E(r)$ for some $E\subset\R^n$ are the two straight lines $F(r)=0$ (already covered by the previous case), and the straight line $F(r)=nr$, with $E$ given either by the compact set of Corollary~\ref{cor:Compact} or the non-compact set of Theorem~\ref{thm:OpenMinusDense}.

\begin{figure}[htb!]  \centering
\begin{tikzpicture}[scale=1.5]
\fill [fill=gray!20](0,0) -- (0,-2) -- (1,0) -- (1,2) -- (0,0);
\draw [very thick](0,0) -- (0,-2) -- (1,0) -- (1,2) -- (0,0);
\draw [->](-0.5,0) -- (1.5,0); \draw (2,0) node {$r=1/p$};
\draw [->](0,-0.5) -- (0,2.2);   \draw (-0.15,2.2) node {$s$};
\draw [thick,dashed](1,0) -- (0,-1.2619);
\draw [thick, dotted](0,0) -- (0.5,0.4) -- (1,0.4);
\draw (1.3,0.4) node {$s_{\alpha,2}$};
\draw (1.2,2) node {$n$};
\draw (-0.3,-2) node {$-n$};
\draw (-0.4,-1.3) node {$d-n$};
\draw (-0.2,-0.2) node {$0$};
\draw (0.5,-0.2) node {$\frac12$};
\draw (1.1,-0.2) node {$1$};
\end{tikzpicture}
\caption{Schematic showing the region in the $rs$-plane in which the graph of $S_E(r)=s_E(1/r)$ must lie.
The dashed line in the lower triangle is the graph of $S_E$ for a set $E$ with Hausdorff dimension $0<d<n$, for instance a Cantor set from Theorem~\ref{thm:CantorZoo}. 
The dotted line in the upper triangle represents a lower bound $F(r)=\min\{2r s_{\alpha,2}, s_{\alpha,2}\}$ for $S_E$ for a fat Cantor set $E=G^{(n)}_{\alpha,\beta}$ from Example~\ref{ex:FatCantor} and Proposition~\ref{prop:FatCantor}.
}\label{fig:Parallelogram}\end{figure}
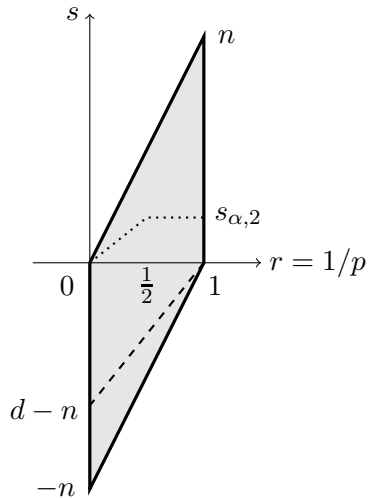

\paragraph{\textbf{Q3:} Under what conditions on $E$ and $p$ is $E$ ``threshold null'' (i.e.\ $(s_E(p),p)$-null)?}
For the case $m(E)=0$ we gave a complete classification of the possible threshold nullity behaviours in Corollary \ref{cor:NullitySets}, which we demonstrated was sharp by providing a zoo of examples in Theorem~\ref{thm:CantorZoo}.
The only other general results we know of concerning threshold nullity are the following:
\begin{itemize}
\item If $E$ is countable then it is $(s_E(p),p)$-null for $1<p<\infty$, with $s_E(p)=-n/p'$ (cf.\ Corollary~\ref{cor:Countable}).
\item If $E=A\setminus Q$, where $A$ has non-empty interior in which the countable set $Q$ is dense, then $E$ is not $(s_E(p),p)$-null for $1<p<\infty$, with $s_E(p)=n/p$ (cf.\ Theorem~\ref{thm:OpenMinusDense}).
\item If $E$ is a compact $d$-set or a $d$-dimensional hyperplane for $0<d<n$, then it is $(s_E(p),p)$-null for $1<p<\infty$, with $s_E(p)=(d-n)/p'$ (cf.\ Theorem \ref{thm:dSet}).
\item If $m(E)=0$ and $\dim_H E=n$, then $E$ is $(s_E(p),p)$-null for $1<p<\infty$, with $s_E(p)=0$.
\item If $s_E(p)=n/p$, as for the sets in Corollary~\ref{cor:Compact} and Theorem~\ref{thm:OpenMinusDense}, then if $E$ is $(s_E(p_0),p_0)$-null for some $1<p_0<\infty$ then $E$ is also $(s_E(p_1),p_1)$-null for every $1<p_1<p_0$.
\end{itemize}

\begin{table}[htb]
\caption{A summary of our results.
The first part of the table reviews the dependence of $s_E(p)$ on general properties of $E$, and the second part catalogues the examples introduced in \S\ref{subsec:SwissCheese} and \S\ref{subsec:ResultsSG0}.}
\begin{center}
\begin{tabular}{|l|l|l|l|}
\hline
Conditions on $E\subset\R^n$ ($E\ne\emptyset$, Borel)
& Nullity threshold $s_E(p)$ & Threshold-null?&\\
\hline
$E$ with non-empty interior & $s_E(p)=+\infty$&&\\
$E$ with empty interior & $-n/p'\le s_E(p)\le n/p$&&\\
Countable $E$  & $s_E(p)=-n/p'$& Yes&Cor.~\ref{cor:Countable}\\
$0\leq \dimH E<n$ & $s_E(p)=(\dimH E-n)/p'$ && Th.\ \ref{thm:NullityHausdorff}\\
$d$-set, Lipschitz $d$-dim.\ manifold &$s_E(p)=(d-n)/p'$ &Yes & Th.\ \ref{thm:dSet}\\
$E=\deO$, with $\Omega$ open, int$(\Omega^c)\ne\emptyset$ & $-1/p'\le s_E(p)\le 0$ && 
Th.\ \ref{thm:Domains}\\
$E=\deO$, with $\Omega$ open $C^{0,\alpha}$ & $-1/p'\le s_E(p)\le -\alpha/p'$ && 
Th.\ \ref{thm:Domains}\\
$E=\deO$, with $\Omega$ open Lipschitz & $s_E(p)= -1/p'$ &Yes &Th.\ \ref{thm:Domains}\\
$\dimH(E)=n$, $m(E)=0$ & $s_E(p)=0$ & Yes & 
Th.\ \ref{thm:NullityHausdorff}\\
$m(E)>0$&$s_E(p)\ge0$&&Lem.\ \ref{lem:nullity1}\\
$E=$ open $\setminus$ countable dense 
& $s_E(p)=n/p$ & No & Th.~\ref{thm:OpenMinusDense}\\
\hline\hline
Examples &&&\\
\hline
Cantor set $E^{(n)}_{d,\pp}$, $1<\pp<\infty$ &$s_E(p)=(d-n)/p'$ &Only if $p\le\pp$&Th.\ \ref{thm:CantorZoo}
\\
Cantor set $F^{(n)}_{d,\pp}$, $1\le\pp\le\infty$ &$s_E(p)=(d-n)/p'$ &Only if $p<\pp$&Th.\ \ref{thm:CantorZoo}
\\
Fat Cantor set $G_{\alpha,\beta}^{(n)}$ & $s_E(p)\ge \min\{s_{\alpha,p},s_{\alpha,2}\}$&&
Prop.\ \ref{prop:FatCantor}
\\
Super-fat Cantor set $K_{\gamma,\delta}^{(n)}$ & $s_E(p)\ge \min\{1/p,1/\pp\}$& No for $p\ge\pp$&Prop.\ \ref{prop:FatCantorCapacity}
\\
Polking set $K_{\pp}$ & $s_E(p)\ge\min\{n/p, n/\pp\}$ & No for $p\ge \pp$
&Th.\ \ref{thm:PolkingCheese}\\
Union of Polking sets & $s_E(p)=n/p$ & No & Cor.\ \ref{cor:Compact}\\
$\R^n\setminus\Q^n$ & $s_E(p)=n/p$ & No & Th.\ \ref{thm:OpenMinusDense}\\
\hline
\end{tabular}
\end{center}
\label{tab:Summary}
\end{table}

\section*{Acknowledgements}
The authors thank Simon Chandler-Wilde and Markus Hansen for helpful discussions.

\appendix
\section{Tensor products and traces}
\label{s:TensorProduct}

In this section we collect some results on the Sobolev regularity of tensor-product distributions.
Propositions~\ref{prop:TensorPositive} and \ref{prop:TensorNegative} give lower and a upper bounds, respectively, for the regularity of a tensor-product distribution in terms of the regularity of the factors.
They are simple consequences of the results proved in \cite{HansenPhD} in the more general setting of Triebel--Lizorkin spaces and the so-called spaces of dominating mixed smoothness.
We improve them slightly in the Hilbert space case $p=2$ by exploiting Plancherel's theorem.
The consequences of these results for $(s,p)$-nullity are summarised in Proposition~\ref{cor:TensorProduct} and in the subsequent paragraph.
In studying $(s,p)$-nullity of Cartesian-product sets, we also make use of classical trace operators and standard bounds on the Hausdorff dimensions of products: see Proposition~\ref{prop:Trace}. 
Note that we do not consider ``mixed'' tensor products of Sobolev distributions with different integrability parameter $p$, so we fix $1<p<\infty$ for the whole section.

Given $n_1,n_2\in\N$ and two distributions $u_1\in \scrD^*(\R^{n_1})$, $u_2\in \scrD^*(\R^{n_2})$, we define the tensor product $u_1\otimes u_2\in\scrD^*(\R^{n_1+n_2})$ as in 
\cite[Chapter V]{HormanderI90} or \cite[Proposition~1.3.1]{HansenPhD}.
This definition immediately extends to tensor products of finitely many distributions $u_1\otimes\cdots\otimes u_N\in\cD^*(\R^{n_1+\cdots n_N})$.

\begin{prop}\label{prop:TensorPositive}
Let $N\in\N$, $s_j\in\R$,  $n_j\in\N$, 
and $u_j\in H^{s_j,p}(\R^{n_j})$ for $j=1,\ldots,N$.
Then
\begin{align}\label{eq:MarkusPositive}
u_1\otimes\cdots\otimes u_N\in \Hsp(\R^{n_1+\cdots+n_N}),
\quad\text{for}\; s<s_*,
\end{align}
where
\begin{align*}
s_*:=\max\big\{0,\min_{j=1,\ldots,N}s_j\big\}+\sum_{j=1}^N \min\{0,s_j\}
=\min_{\emptyset\ne \mathfrak J\subset\{1,\ldots,N\}}\sum_{j\in \mathfrak J} s_j
=\begin{cases}
\displaystyle
\min_{j=1,\ldots,N} s_j &\text{if }s_j\ge 0 \; \forall j,\\[3mm]
\displaystyle
\sum_{j \text{ s.t.\ } s_j<0} s_j &\text{otherwise.}
\end{cases}
\nonumber
\end{align*}
If either $p=2$ or $s_j\in\N_0$ for all $1\le j\le N$, then \eqref{eq:MarkusPositive} holds also for $s=s_*$.
\end{prop}
\begin{proof}
Assertion \eqref{eq:MarkusPositive} is proved in \cite[Propositions~2.3.8(ii) and 4.4.1]{HansenPhD}.
The stronger result when $s_j\in\N_0$ for all $1\leq j\leq N$ follows from the equivalence in this case between the norm in $H^{s_j,p}(\R^{n_j})$ and norm in $W^{s_j,p}(\R^{n_j})$ (involving $L^p$ norms of weak derivatives), see  \cite[Proposition~2.3.8(i)]{HansenPhD}.

To prove the stronger result for $p=2$ it suffices to consider the case with $N=2$ components.
In this case, $s_*=\min\{s_1,s_2,s_1+s_2\}$ and the squared norm $\|u_1\otimes u_2\|_{H^{s_*,2}(\R^{n_1+n_2})}^2$ can be controlled using the representation \eqref{eqn:Hs2NormFourier}, the relation $\widehat{u_1\otimes u_2}=\hat u_1\otimes\hat u_2$, Fubini's theorem, and by bounding the function
$(1+|\bxi_1|^2+|\bxi_2|^2)^{s_*}$ for all $\bxi_1\in\R^{n_1},\bxi_2\in\R^{n_2}$, using the inequalities
\begin{align*}
(1+|\bxi_1|^2+|\bxi_2|^2)^{s_1}
\le(1+|\bxi_1|^2)^{s_1}(1+|\bxi_2|^2)^{s_1}
&\le(1+|\bxi_1|^2)^{s_1}(1+|\bxi_2|^2)^{s_2}
\qquad\text{for } 0\le s_1\le s_2,\\ 
(1+|\bxi_1|^2+|\bxi_2|^2)^{s_1}
\le(1+|\bxi_1|^2)^{s_1}
&\le(1+|\bxi_1|^2)^{s_1}(1+|\bxi_2|^2)^{s_2}
\qquad\text{for }s_1<0\le s_2, \\ 
(1+|\bxi_1|^2+|\bxi_2|^2)^{s_1+s_2}
&=(1+|\bxi_1|^2+|\bxi_2|^2)^{s_1}(1+|\bxi_1|^2+|\bxi_2|^2)^{s_2}\\
&\le(1+|\bxi_1|^2)^{s_1}(1+|\bxi_2|^2)^{s_2} 
\qquad\text{for }s_1,s_2<0.
\end{align*}
This gives $\|u_1\otimes u_2\|_{H^{s_*,2}(\R^{n_1+n_1})} \le \|u_1\|_{H^{s_1,2}(\R^{n_1})}\, \|u_2\|_{H^{s_2,2}(\R^{n_2})}$,  from which the assertion follows.
\end{proof}

\begin{prop}\label{prop:TensorNegative}
Let $N\in\N$, $s\in\R$, 
$n_j\in\N$, and $0\ne u_j\in \cS^*(\R^{n_j})$ for $j=1,\ldots,N$.
If $u_1\notin \Hsp(\R^{n_1})$, then  $u_1\otimes \cdots\otimes u_N\notin H^{t,p}(\R^{n_1+\cdots+n_N})$ for any $t>s$.
If either $p=2$ or $s\in\N_0$, then  $u_1\otimes \cdots\otimes u_N\notin \Hsp(\R^{n_1+\cdots+n_N})$.
\end{prop}
\begin{proof}
Propositions 2.3.8(ii) and 4.4.1 in \cite{HansenPhD} show that, given $s_j\in\N$, $j=1,\ldots,N$, 
$H^{t,p}(\R^{n_1+\cdots+n_N})$ $\subset H^{s_1,p}(\R^{n_1})\otimes_{\alpha_p}\cdots \otimes_{\alpha_p} H^{s_N,p}(\R^{n_N})$ 
(defined as the completion of the algebraic tensor product space under the norm $\alpha_p$ of \cite[Definition~1.3.1]{HansenPhD})
for $t>s^*$, where
\begin{align*}
s^*:=\min\big\{0,\max_{j=1,\ldots,N}s_j \big\}+\sum_{j=1}^N \max\{0,s_j\}
=\max_{\emptyset\ne \mathfrak J\subset\{1,\ldots,N\}}\sum_{j\in \mathfrak J} s_j
=\begin{cases}
\displaystyle
\max_{j=1,\ldots,N} s_j &\text{if }s_j\le 0 \; \forall j,\\[3mm]
\displaystyle
\sum_{j \text{ s.t.\ } s_j>0} s_j &\text{otherwise.}
\end{cases}
\end{align*}
To prove the assertion it suffices to consider the case $N=2$, in which case $s^*=\max\{s_1,s_2,s_1+s_2\}$.
Setting $s_1=s$, it follows that if $u_1\otimes u_2\in H^{t,p}(\R^{n_1+n_2})$ then $u_1\otimes u_2\in H^{s,p}(\R^{n_1})\otimes_{\alpha_p} H^{s_2,p}(\R^{n_2})$ for sufficiently small $s_2$ (i.e. $s_2\le\min\{0,s\}$, so that $s_*=s$). 
Since $\alpha_p$ is a so-called crossnorm \cite[equation~(1.3.2)]{HansenPhD}, this in turn implies 
that $u_2\in H^{s_2,p}(\R^{n_2})$, and more importantly $u_1\in \Hsp(\R^{n_1})$, which proves the assertion by contrapositive. 
The stronger result for the case $s\in\N_0$ comes from \cite[Proposition~2.3.8(i)]{HansenPhD}.

To prove the stronger assertion for the case $p=2$, we note that since $u_2\in H^{s_2,2}(\R^{n_2})$ it holds that $\hat u_2\in L^2_{loc}(\R^{n_2})$. 
Since $\hat u_2\neq 0$, there exist $c_0>0$ and a bounded measurable set $A\subset\R^{n_2}$ with positive measure $m(A)$ such that $|\hat u_2(\bxi_2)|^2\ge c_0>0$ for a.e.\ $\bxi_2\in A$.
Then, if $u_1\otimes u_2\in H^{s,2}(\R^{n_1+n_2})$, the Plancherel and Fubini theorems give the following contradiction:
\begin{align*}
\infty>\|u_1\otimes u_2\|^2_{H^{s,2}(\R^{n_1+n_2})}
&=\int_{\R^{n_1}}\int_{\R^{n_2}}|\hat u_1(\bxi_1)|^2|\hat u_2(\bxi_2)|^2(1+|\bxi_1|^2+|\bxi_2|^2)^s 
\rd\bxi_2\rd\bxi_1\\
&\ge
c_0\int_{\R^{n_1}}\int_{A}|\hat u_1(\bxi_1)|^2(1+|\bxi_1|^2+|\bxi_2|^2)^s \rd\bxi_2\rd\bxi_1\\
&\ge\left\{
\begin{aligned}
&c_0\int_{\R^{n_1}}\int_{A}|\hat u_1(\bxi_1)|^2(1+|\bxi_1|^2)^s \rd\bxi_2\rd\bxi_1,
&& s\ge0,\\
&c_0\int_{\R^{n_1}}|\hat u_1(\bxi_1)|^2(1+|\bxi_1|^2)^s \rd\bxi_1\int_{A}(1+|\bxi_2|^2)^s\rd\bxi_2,
\qquad&&s<0,
\end{aligned}\right.\\
&\ge\left\{
\begin{aligned}
&c_0\;m(A)\;\|u_1\|^2_{H^{s,2}(\R^{n_1})}
&\qquad=\infty,
\qquad&\; s\ge0,\\
&c_0 \;m(A)\; \Big(1+\sup_{\bxi_2\in A}|\bxi_2|^2\Big)^s 
\|u_1\|^2_{H^{s,2}(\R^{n_1})}
&\qquad=\infty,
\qquad& \;s<0.
\end{aligned}\right.
\end{align*}
\end{proof}
To better interpret these results, we define the ``maximal Sobolev regularity'' of a distribution:
$$ m_{p,n}(u):=\sup\big\{s\in \R,\text{ such that } u\in \Hsp\Rn\big\}\in \R\cup\{\pm\infty\}, \quad u\in\cS^*\Rn,\; 1<p<\infty,\;n\in\N.$$
Then, Propositions~\ref{prop:TensorPositive} and \ref{prop:TensorNegative} combine to give a precise characterisation of the maximal Sobolev regularity of a tensor-product distribution if the maximal Sobolev regularity of at least one of the two factors is non-negative:
\begin{align*}
&n_j\in\N, \quad 
\max\big\{m_{p,n_1}(u_1),m_{p,n_2}(u_2)\big\}\ge 0
\quad\Rightarrow
\quad
m_{p,n_1+n_2}(u_1\otimes u_2) = \min\big\{m_{p,n_1}(u_1),m_{p,n_2}(u_2)\big\}.
\end{align*}
Moreover, if $p=2$ or $m_{p,n_1+n_2}(u_1\otimes u_2)\in \N_0$, then $u_1\otimes u_2$ belongs to $H^{m_{p,n_1+n_2}(u_1\otimes u_2),p}(\R^{n_1+n_2})$ if and only if  
$u_1 \in H^{m_{p,n_1+n_2}(u_1\otimes u_2),p}(\R^{n_1})$ and $u_2 \in H^{m_{p,n_1+n_2}(u_1\otimes u_2),p}(\R^{n_2})$. 
A similar statement holds for the tensor product of $N$ distributions, if all except at most one have non-negative maximal Sobolev regularity.
If both $u_1$ and $u_2$ have negative maximal Sobolev regularity, the result is less sharp:
\begin{align*}
&n_j\in\N, \quad 
\max\big\{m_{p,n_1}(u_1),m_{p,n_2}(u_2)\big\}\le 0\quad\\
&\Rightarrow\quad
m_{p,n_1}(u_1)+m_{p,n_2}(u_2)\;\le\;
m_{p,n_1+n_2}(u_1\otimes u_2)\;\le\; \min\big\{m_{p,n_1}(u_1),m_{p,n_2}(u_2)\big\}.
\end{align*}
We point out that the lower bound here can be achieved. 
Indeed, if $\bx_j\in\R^{n_j}$, $j=1,2$, then by \eqref{eq:delta} $m_{p,n_j}(\delta_{\bx_j})=-n_j/p'$, $j=1,2$, and $m_{p,n_1+n_2}(\delta_\bx\otimes\delta_\by)=-(n_1+n_2)/p'=m_{p,n_1}(\delta_\bx)+m_{p,n_2}(\delta_\by)$.

So far we have related distributions defined on Euclidean spaces with different dimensions using tensor products. 
To relate functions (with positive regularity exponent $s$) defined on an ambient Euclidean space and on affine subspaces one can use \emph{traces}.
Using classical results on traces in Triebel--Lizorkin spaces we can prove the following result. 
\begin{prop}\label{prop:Trace}
Let $n_1,n_2\in\N$,  $1<p<\infty$, $s\in\R$, $E_1\subset\R^{n_1}$ and $E_2\subset \R^{n_2}$. If $E_1$ is $(s,p)$-null then
$E_1\times E_2$ is $(t,p)$-null for 
$$\begin{cases}
t\ge s+\frac{n_2}p & \text{ if } s>0,\; 1<p\le2,\\
t>s+\frac{n_2}p & \text{ if } s>0,\;2<p<\infty,\\
t\ge s & \text{ if } s=0,\\
t>s & \text{ if } s<0 \text{ and $E_1,E_2$ Borel}.
\end{cases}$$
\end{prop}
\begin{proof}
The case $s=0$ follows from Lemma~\ref{lem:nullity1}\rf{qq}, while the case $s<0$ can easily be deduced from the relation between nullity and Hausdorff dimension 
in \eqref{eq:DimHCharac} and the inequality $\dimH(E_1\times E_2)\le \dimH(E_1)+n_2$ \cite[equation~(7.7)]{Fal}.

For the case $s>0$, for any $\by\in\R^{n_2}$, \S2.7.2 and \S2.3.2 of \cite{Triebel83ThFS} give continuity of the trace operator 
$$\Tr_\by: H^{t,p} (\Rnn)\to H^{s,p}(\R^{n_1}), \qquad s>0,\; 
\begin{cases}
t\ge s+\frac{n_2}p & \text{ if } 1<p\le2,\\ 
t> s+\frac{n_2}p & \text{ if } 2<p<\infty,
\end{cases}
$$
defined on $\cD(\Rnn)$ as the pointwise trace onto the subspace $\{(\bx_1,\bx_2)\in\Rnn,\bx_2=\by\}$ (which is canonically identified with $\R^{n_1}$), and then extended to $H^{t,p}(\Rnn)$ by density. 
If $0\ne u\in H^{t,p}(\Rnn)\subset L^p(\Rnn)$, by Fubini's theorem there exists at least one $\by\in\R^{n_2}$ such that $\Tr_\by (u)\ne 0$.
If moreover $\supp u\subset E_1\times E_2$, by convolution with a sequence of mollifiers in $\cD(\Rnn)$ with decreasing support it is straightforward to show that $\supp(\Tr_\by(u))\subset E_1$.
This proves the assertion by contrapositive.
\end{proof}

\section{A result on the non-equality of capacities}
\label{app:CapacityEquivalence}
In this appendix 
(due to Simon Chandler-Wilde) we give a concrete example of an open set $\Omega\subset \R$ for which $\ccap_{2,2}(\Omega)< \cCap_{2,2}(\Omega)$. 
Specifically, we consider an open interval $\Omega=(-a,a)\subset \R$, where $a>0$. 
We recall that 
the space $H^{2,2}(\Omega)\subset \scrD^*(\Omega)$ is defined by 
\begin{align*}
H^{2,2}(\Omega) &:=\big\{u\in \scrD^*(\Omega): u=U|_\Omega \textrm{ for some } U\in H^{2,2}\big\},\\
\|u\|_{H^{2,2}(\Omega)} &:= \inf_{\substack{U\in H^{2,2}\\ U|_\Omega=u}} \|U\|_{H^{2,2}}.
\end{align*}
It is straightforward to show that 
\begin{align}
\label{capHsA}
\ccap_{2,2}(\Omega) &= \inf\{\|u\|_{H^{2,2}(\Omega)}^2:u\in H^{2,2}(\Omega), u\geq 1 \textrm{ a.e.\ on } \Omega\},\\
\label{CapHsA}
\cCap_{2,2}(\Omega) &= \inf\{\|u\|_{H^{2,2}(\Omega)}^2:u\in H^{2,2}(\Omega), u=1 \textrm{ a.e.\ on } \Omega\} =\|1\|_{H^{2,2}(\Omega)}^2.
\end{align}

An explicit formula for $\|u\|_{H^{2,2}(\Omega)}$ in the case where $\Omega$ is an open interval has been given recently in \cite[Equation (26)]{InterpolationCWHM}. For even functions $u\in H^{2,2}(\Omega)$ this formula gives (note that we correct a typographical error in \cite[Equation (26)]{InterpolationCWHM}, replacing $-\phi'(a)$ in that formula by $+\phi'(a)$)
\begin{align}
\label{eqn:H2Explicit}
\|u\|_{H^{2,2}(\Omega)}^2 = 2\left( |u(a)|^2 + |u'(a)|^2 + |u(a)+u'(a)|^2 + \int_0^a (|u(t)|^2 + 2|u'(t)|^2+ |u''(t)|^2) \, \rd t\right).
\end{align}
From \rf{eqn:H2Explicit} and \rf{CapHsA} it follows that $\cCap_{2,2}(\Omega)=\|1\|_{H^{2,2}(\Omega)}^2 = 4+2a$. But using \rf{eqn:H2Explicit} we can for any $a>0$ construct a function $u\in H^{2,2}(\Omega)$ which satisfies $u\geq 1$ on $\Omega$ and has $\|u\|_{H^{2,2}(\Omega)}^2 < 4+2a$. This, in the light of \rf{capHsA}, demonstrates that $\ccap_{2,2}(\Omega)<\cCap_{2,2}(\Omega)$ for this particular $\Omega$. 

When $a<\sqrt{3}$ we consider the quadratic function $u(t)=1+\eps(a^2-t^2)$, for some $\eps>0$ to be specified. By \rf{eqn:H2Explicit} this function satisfies
\begin{align*}
\|u\|_{H^{2,2}(\Omega)}^2 = 4+ 2a - 8a\left(1-\frac{a^2}{3}\right)\eps + \ord{\eps^2}, \qquad \eps\to 0,
\end{align*}
so that $\|u\|_{H^{2,2}(\Omega)}^2 < 4+2a$ for sufficiently small $\eps$, provided $a<\sqrt{3}$.

When $a>1$ we consider the function (again with $\eps>0$ to be specified)
\begin{align*}
u(t)=
\begin{cases}
1, & |t|\leq a-1,\\
1+\eps (a-|t|)(a-1-|t|)^2, & a-1\leq |t|<a,
\end{cases}
\end{align*}
which by \rf{eqn:H2Explicit} satisfies
\[ \|u\|_{H^{2,2}(\Omega)}^2=4+2a-\frac{11}3 
\eps + \ord{\eps^2}, \qquad \eps\to 0, \]
and again we have $\|u\|_{H^{2,2}(\Omega)}^2 < 4+2a$ for sufficiently small $\eps$.


\end{document}